\documentclass[11pt,reqno,a4paper]{amsart}
\usepackage{amsfonts}
\usepackage{amsmath, amscd,amssymb,bm,amsbsy,epsf}
    \usepackage{enumerate}
  \usepackage{paralist}
\usepackage{bbm, dsfont}
\usepackage{graphicx,color}
\usepackage{subfigure}
\usepackage{mathrsfs}
\usepackage{verbatim}
\usepackage{inputenc}
\usepackage{bbm}
\usepackage{multicol}
\usepackage{balance}
\usepackage{verbatim}
\usepackage{accents}
\usepackage{enumitem}

\newtheorem{definition}{Definition}

\newtheorem{theorem}{Theorem}
\newtheorem{lemma}[theorem]{Lemma}

\newtheorem{proposition}[theorem]{Proposition}
\def\N{\bm{\mathbb{N} } }
\def\R{\mathbb{R}}

\newcommand{\la}{\lambda}

\newcommand{\red}{\textcolor{red}}

\def\ep{\varepsilon}

\def\NN{\mathcal{N}}
\def\MM{\mathcal{M}}
\numberwithin{equation}{section}

\title[]{Existence and Morse Index of least energy nodal solution of the $(p,2)$-laplacian}

\author[O. Agudelo]{Oscar Agudelo}
\address{\noindent O. Agudelo - NTIS, Department of Mathematics,
Zapadoceska Univerzita v Plzni, Plzen, Czech Republic.}
\email{oiagudel@ntis.zcu.cz}

\author[D. Restrepo]{Daniel Restrepo}
\address{\noindent D. Restrepo - Department of Mathematics, University of Texas at Austin, Austin TX, United States.}
\email{daerestrepomo@utexas.edu}

\address{\noindent D. Restrepo - Escuela de Matem\'{a}ticas, Universidad Nacional de Colombia Sede Medell\'{i}n, Medell\'{i}n, Colombia.}

\author[C. V\'{e}lez]{Carlos V\'{e}lez}
\address{\noindent C. V\'{e}lez - Escuela de Matem\'{a}ticas, Universidad Nacional de Colombia Sede Medell\'{i}n, Medell\'{i}n, Colombia.}
\email{cauvelez@unal.edu.co}

\begin{document}
\maketitle

\begin{abstract}
In this paper we study the quasilinear equation $- \ep^2 \Delta u-\Delta_p u=f(u)$ in a smooth bounded domain $\Omega$ with Dirichlet boundary condition. For $\ep \geq 0$, we review existence of a least energy nodal solution and then present information about the Morse Index of least nodal energy solutions this BVP. In particular we provide Morse Index information for the case $\ep =0$. 
\end{abstract}
%
%
%\begin{keyword}
%p-Laplacian \sep least energy nodal solutions \sep Morse index \sep  Nehari manifold.

%\MSC[2010] 35A01\sep 35A15 \sep 35B38\sep 35J05\sep 35J60\sep 35J62\sep 35J66\sep 35J92.
%\end{keyword}

\section{Introduction}

\label{Introduction}
Consider the {\it boundary value problem} (BVP):
\begin{equation}\label{eq mean problem}\tag{$\mathcal{P}_\varepsilon$}
\left\{
\begin{aligned}
- \ep^2 \Delta u-\Delta_p u=&f(u) &\hbox{in} &\quad \,\,\,\Omega, \,\\
u =& 0     & \hbox{on} & \quad \partial \Omega, 
\end{aligned}
\right.
\end{equation}
where $\ep \in \R$ is a parameter, $\Omega\subset \R^N$, $N \geq 1$, is a smooth bounded domain when $N \geq 2$ and an open bounded interval when $N=1$. We assume that $p\in (2,+ \infty)$ and we denote by $\Delta u:= {\rm div}(\nabla u)$ and $\Delta_p :={\rm div}(|\nabla u|^{p-2}\nabla u)$ the Laplace and the $p$-Laplace operators of a function $u$, respectively.

\medskip
In what follows, we set  
$$
p^*:=\left\{\begin{array}{cc}
\frac{Np}{N-p},& \quad \hbox{if} \quad p\in (2,N),\\
+ \infty,& \quad \,\,\,\,\hbox{if} \quad p \in [N, \infty) 
\end{array}
\right.
$$ 
and notice that $p\in (2,p^*)$.

\medskip
As for the nonlinearity $f:\R \to \R$, define $F(t):= \int_{0}^t f(s)ds$ for $t\in \R$ and consider the following set of hypotheses. 
\begin{enumerate}[label=({f}{\arabic*})]

\item $f\in C^1(\R)$ and there exist $q\in (p,p^*)$ and $A>0$ such that for every $t\in \R$,
\begin{equation}\label{subcriticalgrowthf}
|f'(t)|\leq A(1+|t|^{q-2}).
\end{equation}

\item  There exist $m\in (p,p^*)$ and $T>0$ such that for every $t\in \R$ with $|t|\geq T$, 
		\begin{equation}\label{p-1superlineargrowth}
		f(t)t \geq m F(t)>0.
		\end{equation}

%\blue{If (f1) and (f2) are assumed, then $m\in (p,q)$.}

\item $f'_p(0):= \limsup \limits_{t\to 0}\frac{f(t)}{|t|^{p-2}t} \in (-\infty, \lambda_{1,p})$, where $\lambda_{1,p}>0$ is the first eigenvalue of $-\Delta_{p}$ with homegeneous Dirichlet boundary condition.

%\blue{Here I suggest removing the definition $f'_p(0)$ or denoting it differently, since a derivative is a limit and the reader could be confused by the limsup whilst reading further into the draft.}

\item[(f3)'] $f(0)=0$.
\medskip

\item The function $
\R-\{0\} \ni t \mapsto \frac{f(t)}{|t|^{p-1}}$ is strictly increasing or equivalently, for every $t\in \R-\{0\}$, 
\begin{equation*}%\label{Monotquotientf}
f'(t)>(p-1)\frac{f(t)}{t}.
\end{equation*}
\end{enumerate}

Notice that (f3) implies (f3)'. \\

For any function $v:\Omega \to \R$, define $v^+(x):=\max\{0,v(x)\}$ and $v^-(x):= \min\{0,v(x)\}$ for a.e. $x\in \Omega$ and any connected component of the set $\{v \neq 0\}$ will be called a {\it nodal domain} or {\it nodal region} of the function $v$.

\medskip

The concepts of solutions we will be working with throughout this paper are next defined. Let $\ep \in \R $ and assume that $f$ satisfies hypothesis (f1).  

\begin{definition}\label{solution}
A solution of \eqref{eq mean problem} is a function $u\in W_0^{1,p}(\Omega)$ such that for every $\varphi \in W_0^{1,p} (\Omega)$,
\begin{equation}\label{eq weak mean problem}
\int_{\Omega}\left( \ep^2 + |\nabla {u}|^{p-2}\right)\nabla {u}\cdot\nabla \varphi dx  =\int_{\Omega} f({u})\varphi dx.
\end{equation}

A {\it nodal solution} of \eqref{eq mean problem} is a solution $u$ such that $u^+,u^-\neq 0$ a.e. in $\Omega$.  
\end{definition}

The energy functional $J_{\ep}:W^{1,p}_0(\Omega) \to \R$, associated with \eqref{eq weak mean problem} is defined by
$$
J_{\ep}(v):= \int_{\Omega} \left(\frac{\ep^2}{2}|\nabla v|^2 + \frac{1}{p}|\nabla v|^p - F(v)\right)dx \quad \hbox{for} \quad v\in W^{1,p}_0(\Omega).
$$

Since $p>2$ and $f$ satisfies (f1), $J_{\ep} \in C^2(W^{1,p}_0(\Omega))$  (see for instance \cite{CASTORINAESPOSITOSCIUNZI2009} and \cite{CingolaniVannella2003}) with first derivative
\begin{equation*}%\label{DerivJep}
DJ_{\ep}(v)\varphi = \int_{\Omega} \left(\ep^2 + |\nabla v|^{p-2}\right)\nabla v\cdot \nabla \varphi dx - \int_{\Omega} f(v)\varphi dx
\end{equation*}
and second derivative
\begin{multline}\label{2ndDerivJep}
D^2J_{\ep}(v)(\varphi, \psi) = \int_{\Omega} \left(\ep^2 + |\nabla v|^{p-2}\right)\nabla \varphi\cdot \nabla \psi dx \\
+ \int_{\Omega}(p-2)|\nabla v|^{p-4}(\nabla v \cdot \nabla \varphi)(\nabla v\cdot \nabla \psi) dx - \int_{\Omega} f'(v)\varphi \psi dx
\end{multline}
for every $v,\varphi,\psi\in W^{1,p}_0(\Omega)$. Solutions to \eqref{eq mean problem} correspond exactly to the critical points of $J_{\ep}$ in $W^{1,p}_0(\Omega)$.

\medskip
In this work we study existence, qualitative properties and {\it Morse index} computations for least energy nodal solutions of \eqref{eq mean problem}.

\medskip
If $u\in W^{1,p}_0(\Omega)$ is a solution of \eqref{eq mean problem}, its {\it Morse index}, $m_{\ep}(u)$, is defined as the maximal dimension of a linear subspace $E$ of $W^{1,p}_0(\Omega)$ such that 
\begin{equation*}%\label{defMorseIndex}
D^2J_{\ep}(u)(\varphi,\varphi) < 0 \quad \hbox{for} \quad \varphi \in E.
\end{equation*}

Similar questions to the ones addressed in this work were studied in \cite{CASTROCOSSIONEUBERGER1997} in the autonomous case, when $\ep =0$, $p=2$ and under similar hypotheses on the nonlinearity. The non-autonomous case, still for $\varepsilon=0$ and $p=2$, has been treated in \cite{BartschWeth2003} under slightly weaker assumptions.

\medskip
Existence and qualitative properties of nodal solutions to \eqref{eq mean problem} in the case $\ep=0$ and $p>1$ have been studied in \cite{BARTSCHLIU2005}.

\medskip 
Recently, the authors in \cite{BarileDeFigueiredo2014} studied the existence of a least energy nodal solution for a related BVP were a more general class of diffusions are considered and under similar assumptions as (f1)-(f4), but assuming in (f3) that $f'_p(0)=0$ and that as $|t|\to \infty$, $f(t)$ has slower growth rate than $|t|^{q-2}t$. We remark that in \cite{BarileDeFigueiredo2014} no information regarding the Morse index of nodal solutions is provided.

\medskip
The main technique used in the aforementioned works is the Nehari manifold method. Our goal is to adapt this technique to our setting in order to study in more detail nodal solutions of \eqref{eq mean problem} and to extend the results in \cite{CASTROCOSSIONEUBERGER1997,BartschWeth2003, BarileDeFigueiredo2014}.

\medskip
In order to state our main results, we introduce some required terminology. For any $\ep \in \R$, the {\it Nehari Manifold} associated to the energy $J_{\ep}$ is the set
\begin{equation}\label{Neharimanifolddef}
\NN_{\ep}:=\{v\in W^{1,p}_0(\Omega)\,:\, v \neq 0 \quad \hbox{and}\quad DJ_{\ep}(v)v=0\}.
\end{equation}

Observe that $\NN_{\ep}$ contains all the non-zero solutions of \eqref{eq mean problem}. Since we are interested in nodal solutions, we consider also the set
\begin{equation}\label{Hemispheres}
\MM_{\ep}:= \{v\in W^{1,p}_0(\Omega)\,:\, v^+ \in \NN_{\ep}\quad \hbox{and} \quad v^- \in \NN_{\ep}\}.
\end{equation}

Observe also that $\MM_{\ep}\subset \NN_{\varepsilon}$ and that $\MM_{\varepsilon}$ contains all the nodal solutions of \eqref{eq mean problem}. In view of this remark, we say that a solution $u\in W^{1,p}_0(\Omega)$ of \eqref{eq mean problem} is a least energy nodal solution if 
\begin{equation}\label{leastnodalenergyslncondition}
u\in \MM_{\ep} \quad \hbox{and} \quad J_{\ep}(u)= \min \limits_{v\in \MM_{\ep}} J_{\ep}(v).
\end{equation}

\medskip
Our main results read as follows.

\begin{theorem}\label{Theorem1}
Let $f$ satisfy hypotheses (f1)-(f4). For any $\ep \in \R$, there exists a least energy nodal solution $u_{\ep}\in W^{1,p}_0(\Omega)$ of \eqref{eq mean problem} and 
\begin{equation}\label{nodalcharacteerization}
J_{\varepsilon}(u_{\varepsilon}) = \min \limits_{\substack{v\in \MM_{\varepsilon}}} J_{\varepsilon}(v)= \min\limits_{\substack{v\in W^{1,p}_0(\Omega),\\ v^+,v^- \neq 0}} \left(\max \limits_{t,s \geq 0}J_{\varepsilon}(tv^+ + sv^-)\right).
\end{equation}
\end{theorem}

Although the proof of the existence of the solution $u_{\varepsilon}$, stated in Theorem \ref{Theorem1}, is essentially  contained in Theorem 1.1 in \cite{BarileDeFigueiredo2014}, we include the detailed proof since many of the elements in it will be used in the computation of the Morse Index of the solution $u_{\ep}$. 

\medskip
Also, up to our knowledge, the {\it min-max} characterization of the energy level $J_{\varepsilon}(u_{\varepsilon})$ in \eqref{nodalcharacteerization} is new and it is alternative to Nehari manifold approach.

\medskip
We point out that Theorem \ref{Theorem1} still holds for $\varepsilon \neq 0$ replacing (f3) by (f3)' and the hypothesis that $f'(0)<\varepsilon^2 \lambda_1$, where $\lambda_1>0$ is the first eigenvalue of $-\Delta$ with homogeneous Dirichlet boundary condition. We impose (f3) instead of these weaker conditions in order to obtain some uniform estimates which will be used later on.

\medskip
%\blue{I would remove this paragraph or move it to the end of the paper in a separate section for concluding remarks. Also, the hypothesis (f3) concerns another type of nonlinearities, so not sure if weaker is the right word. Finally, I would add to the last sentence "which will be used later" the precise place of the estimates which I think are quite spread throughout the paper.}

\medskip
Our second result is concerned with qualitative properties and the Morse index computation for least energy nodal solutions in the case $\ep\neq 0$.

\begin{theorem}\label{Theorem2}
Let $f$ satisfy (f1), (f2), (f3)' and (f4). For any $\ep\neq 0$ and any local minimizer, $u\in \MM_{\varepsilon}\cap C^1(\overline{\Omega})$, of $J_\varepsilon|_{\MM_\varepsilon}$,
\begin{itemize}
\item[(i)] $m_{\ep}(u) =2$ and

\item[(ii)] $u$ has exactly two nodal domains, i.e. the sets $\{u>0\}$ and $\{u<0\}$ are connected and
$$
\{u \neq 0\} = \{u>0\} \cup \{u<0\}.
$$ 
\end{itemize}

In particular, any least energy nodal solution $u$ of \eqref{eq mean problem} satisfies {\it (i)} and {\it (ii)}.
\end{theorem}

For $\ep=0$ and $p=2$ (the semilinear case), Theorem \ref{Theorem2} was already hinted in \cite{CASTROCOSSIONEUBERGER1997} and proved in \cite{BartschWeth2003}. In fact our proof of Theorem \ref{Theorem2} closely follows the scheme of \cite{BartschWeth2003}.

%\medskip
%\blue{I removed the first part because it was a repetition of what the statement of Theorem 2 says.}

\medskip 
Regarding the case $\ep=0$, our next result extends Theorem 1.1 in \cite{BARTSCHLIU2005} and Theorem 1.1 in \cite{BarileDeFigueiredo2014} by providing an example of a nodal solution of $(\mathcal{P}_0)$ having Morse index two.

\begin{theorem}\label{Theorem3}
Let $f$ satisfy hypotheses (f1)-(f4). Then the BVP $(\mathcal{P}_0)$ has a least energy nodal solution $u_0 \in W^{1,p}_0(\Omega) \cap C^{1}(\overline{\Omega})$ with $m_0(u_0)=2$ and with two exactly two nodal domains.
\end{theorem}

\medskip
If a least energy nodal solutions of $(\mathcal{P}_0)$ has some nondegeneracy, then it can be approximated by local minimizers associated to \eqref{eq mean problem}. This is the content of our next result.

\begin{theorem}\label{Theorem4}
Let $f$ satisfy hypotheses (f1)-(f4) and let  $u\in \MM_0$ be a strict local minimizer of $J_0|_{\MM_0}$. Then, there exists a decresing sequence $\{\varepsilon_n\}_{n\in \N}\subset (0,\infty)$ with $\varepsilon_n\to 0$ and there exists a sequence of functions  $\{u_n\}_{n\in \N}\subset W^{1,p}_0(\Omega)$ such that
\begin{itemize}
	\item[(i)] for every $n\in \N$, $u_n \in \MM_{\varepsilon_n}$ and $u_n$ is a local minimizer of  $J_{\varepsilon_n}|_{\MM_{\varepsilon_n}}$,
	\item[(ii)] $u_n \to u$ strongly in $W_0^{1,p}(\Omega)$ as $n\to \infty$ and 
	\item[(iii)] $m_{\varepsilon_n}(u_n)=2$ and $u_n$ has exactly two nodal domains.
	\item[(iv)] $m_0(u)=2$ and $u$ has exactly two nodal domains.
\end{itemize}
\end{theorem}

\medskip
We believe the property of the least energy nodal solution stated in Theorem \ref{Theorem3} to be generic for such solutions. 
In this regards,  Theorem \ref{Theorem4} provides a partial reciprocal to Theorem \ref{Theorem3}. We remark also that Theorem \ref{Theorem4} is true if the solution $u \in \MM_0$ is a limit point of isolated local minizers of $J_0|_{\MM_0}$ (see Remark at the end of the paper).

\medskip
For $\varepsilon=1$ and $p>2$, S. Cingolani and G. Vanella in \cite{CingolaniVannella2003}, obtained  critical groups estimates at any solution of $(\mathcal{P}_1)$, in the spirit of the generalized Morse lemma and assuming only hypothesis (f1). 

\medskip
Recently the authors in \cite{BOBKOVKOLONITSKII2017}, obtained symmetry results and sign changing properties of least energy nodal solutions for the case $\varepsilon=0$ and $p>1$ under similiar assumption on the nonlinearity. In \cite{BOBKOVKOLONITSKII2017} the authors used the domain derivative method to approach a domain optimization problem.

% Also, in {\bf PAPAGEORGIOU2018}, multiplicity of nodal solutions has been studied...ELABORATE MORE ON THIS. 

%\blue{I have not included yet the Papageorgiou reference.}
\medskip
 
%
%\medskip
%In the same spirit, the authors in \cite{CASTORINAESPOSITOSCIUNZI2009}, develop a spectral theory for the same setting as in \cite{CingolaniVannella2003} and use this theory to compute Morse Index and provide qualitative properties of a solution.

\medskip
The paper is organized as follows. In  section 2 we prove several lemmas that will be useful throught the rest of the paper and we present the proof of the Theorem \ref{Theorem1}. Section 3 is devoted to some regularity results needed for the computation of the Morse index of the solutions for $\ep\neq 0$. In section 4 we give the detailed proof of the Theorem \ref{Theorem2} and in section 5 we handle the limiting case $\ep=0$ and present the proofs of the rest of the Theorems.

\section{Qualitative lemmas and existence}\label{Existence}
In this part, we present the existence of a least energy nodal solution for \eqref{eq mean problem}. Our arguments are motivated by those in \cite{CASTROCOSSIONEUBERGER1997}, but we refer the reader to \cite{BarileDeFigueiredo2014}, \cite{BARTSCHLIU2005} and \cite{BartschWeth2003} (and references therein), where existence of least energy nodal solutions was already treated in other settings.

\medskip
We begin with a some remarks that will be crucial throughout this work and introduce also some further notation.

\medskip
Assume hypotheses (f1), (f3). Since $2<p<q$, given any $\mu\in (f'_p(0), \infty)$, there exists $C_{\mu}>0$ such that for any $t\in \R$,
\begin{equation}\label{estimateimportant1}
f(t)t \leq \mu |t|^{p} + C_{\mu}|t|^{q}.
\end{equation}

Assuming hypothesis (f2) and integrating \eqref{p-1superlineargrowth}, we find constants $a,b>0$ such that 
\begin{equation}\label{p-1superlinear}
F(t) \geq a|t|^m - b \quad \hbox{for} \quad t\in \R.
\end{equation}
%\blue{Here I removed the subsequent paragraph, since it was pointless.}

%and hence using again \eqref{p-1superlineargrowth} and that $m> p$ , 
%$$
%\begin{aligned}
%\liminf\limits_{\vert t\vert \to\infty} \frac{f(t)}{|t|^{p-2}t} = &\liminf\limits_{\vert t\vert \to\infty} \frac{tf(t)}{|t|^{p}} \\
%\geq &\liminf\limits_{\vert t\vert \to\infty} \frac{m F(t)}{|t|^{p}},
%\end{aligned}
%$$
%so that 
%\begin{equation*}
%\lim\limits_{\vert t\vert \to\infty} \frac{f(t)}{|t|^{p-2}t}= \infty.
%\end{equation*}

\medskip
Let $\ep \in \R$ be arbitrary, but fixed and consider the function $\gamma_{\ep}: W^{1,p}_0(\Omega) \to \R$ defined by 
$$
\begin{aligned}
\gamma_{\ep}(v):= &DJ_{\ep}(v)(v) \\
=&\int_{\Omega}\left(\ep^2|\nabla v|^2 + |\nabla v|^p - f(v)v\right)dx
\end{aligned}
$$ 
for $v\in W^{1,p}_0(\Omega)$. Observe that $\gamma_{\ep}\in C^1(W^{1,p}_0(\Omega))$ with derivative
\begin{equation}\label{derivgamma}
D\gamma_{\ep}(v)\varphi = \int_{\Omega} \left(\left[2 \ep^2  + p |\nabla v|^{p-2}\right]\nabla v \cdot \nabla \varphi - f'(v)v\varphi - f(v)\varphi\right)dx
\end{equation}
for $v, \varphi \in  W^{1,p}_0(\Omega)$. The set $\NN_{\ep}$, defined in \eqref{Neharimanifolddef}, reads as 
$$
\mathcal{N}_{\ep}:= \{v\in W^{1,p}_0(\Omega)-\{0\}\,:\, \gamma_{\ep}(v)=0\}
$$
and $v \in \NN_{\ep}$ if and only if $v\in W^{1,p}_0(\Omega)-\{0\}$ and
\begin{equation}\label{eqnNN}
\int_{\Omega} \left(\ep^2 |\nabla v|^2 + |\nabla v|^p\right)dx = \int_{\Omega}f(v)vdx.
\end{equation}

Next lemma states that the energy $J_{\ep}|_{\NN_{\ep}}$ is uniformly coercive and the sets $\NN_{\ep}$ are uniformly bounded away from zero. 

\begin{lemma}\label{QualitLemma} Let $f$ satisfy (f1)-(f3). Then, the following properties are satisfied.
	\begin{itemize}
		\item[(i)] There exist $C\in \mathbb{R}$ such that for any $\ep \in \R$ and any $v \in \NN_{\ep}$, 
\begin{equation}\label{coercivity}
J_{\ep}(v) \geq \left(\frac{1}{p}- \frac{1}{m}\right)\|v\|_{W^{1,p}_0(\Omega)}^p+C.
\end{equation}
\item [(ii)] There exists $\rho>0$ such that for every $\ep \in \R$ and every $v\in \NN_{\ep}$, 
\begin{equation}\label{boundedawayfromzero}
\|v\|_{W^{1,p}_0(\Omega)} \geq \rho.
\end{equation}
Moreover, there exists  $\rho_0>0$ such that for every $\ep \geq 0$ and every $v\in \NN_{\ep}$, 
\begin{equation}\label{boundedawayfromzerolq}
\|v\|_{L^q(\Omega)} \geq \rho_0,
\end{equation}
where $q$ is the exponent introduced in the condition (f1). 
\item [(iii)] For every $\ep \in \R$, $\NN_{\ep}$ is a closed subset of $W^{1,p}_0(\Omega)$.
\end{itemize}
\end{lemma}

\begin{proof}
Let $\ep \in \R$ and $v\in \NN_{\ep}$. Using (f2) and the fact that $m>p>2$,
$$
\begin{aligned}
J_{\ep}(v) =& J_{\ep}(v) - \frac{1}{m}\gamma_{\ep}(v) \\
=& \ep^2\left(\frac{1}{2}- \frac{1}{m}\right)\int_{\Omega}|\nabla v|^2dx + \left(\frac{1}{p}- \frac{1}{m}\right)\int_{\Omega} |\nabla v|^p dx \\
&\hspace{5cm}- \frac{1}{m}\int_{\Omega}\left(m F(v) - f(v)v\right)dx\\
\geq &   \left(\frac{1}{p}- \frac{1}{m}\right)\int_{\Omega} |\nabla v|^p dx - \frac{1}{m}\int_{\{|v|\leq T\}}\left(m F(v) - f(v)v\right)dx.
\end{aligned}
$$

Since the function $[-T,T]\ni t\mapsto mF(t)-f(t)t$ is bounded, there exists $C\in \R$, depeding only on $f$ and $\Omega$, such that
$$
J_{\ep}(v) \geq \left(\frac{1}{p}- \frac{1}{m}\right)\int_{\Omega} |\nabla v|^p dx
+C$$
and this proves {\it(i)}.

\medskip
To prove {\it (ii)}, we proceed as follows. Let $\mu \in (f'_p(0), \lambda_{1,p})$ be fixed and choose $C_{\mu}>0$ so that \eqref{estimateimportant1} holds. Using that $v\in \NN_{\ep}$,
$$
\int_{\Omega}|\nabla v|^p dx \leq   
\int_{\Omega} \left(\ep^2 + |\nabla v|^{p-2}\right)|\nabla v|^2 dx
= \int_{\Omega}f(v)v dx
$$
and from \eqref{estimateimportant1} we find that
\begin{equation}\label{inequality0}
\int_{\Omega}|\nabla v|^p dx \leq   \mu\int_{\Omega}|v|^p dx + C_{\mu} \int_{\Omega} |v|^q dx. 
\end{equation}

From the Sobolev inequalities, there exists $\tilde{C}_{\mu}$ such that
$$
\int_{\Omega}|\nabla v|^p dx \leq \frac{ \mu}{\lambda_{1,p}}\int_{\Omega}|\nabla v|^p dx + \tilde{C}_{\mu}\left(\int_{\Omega} |\nabla v|^p dx \right)^\frac{q}{p}
$$
so that 
$$
\left(1-\frac{\mu}{\lambda_{1,p}}\right) \leq \tilde{C}_{\mu} \|v\|_{W^{1,p}_0(\Omega)}^{q-p}.
$$

This proves \eqref{boundedawayfromzero} with $\rho = \left(1-\frac{\mu}{\lambda_{1,p}}\right)^{\frac{1}{q-p}}$.

\medskip
To prove \eqref{boundedawayfromzerolq}, we argue by cases. If $\|v\|_{L^q(\Omega)} \geq 1$, then \eqref{boundedawayfromzerolq} is proven with any $\rho_0 \in (0,1]$. If $\|v\|_{L^q(\Omega)}\in (0,1)$, then $\|v\|_{L^q(\Omega)}^q<\|v\|_{L^q(\Omega)}^p$ and using \eqref{inequality0} and H\"{o}lder inequality, 
$$
\rho^p\leq \mu {\rm meas}(\Omega)^{\frac{q}{q-p}}\|v\|_{L^q(\Omega)}^p +  C_{\mu}\|v\|_{L^q(\Omega)}^q \leq \left(\mu + C_{\mu}\right)\|v\|_{L^p(\Omega)}^p,
$$
where ${\rm meas}(\Omega)$ is the Lebesgue measure of $\Omega$. 

\medskip 
This implies the existence of $\rho_0>0$ such that $\Vert v\Vert_{L^q(\Omega)}\geq \rho_0$ proving \eqref{boundedawayfromzerolq} and consequently concluding the proof of {\it (ii)}.

\medskip
Next, we prove {\it (iii)}. Let $\{v_n\}_{n\in \mathbb{N}} \subset \NN_{\ep}$ such that $v_n \to v$ strongly in $W^{1,p}_0(\Omega)$. Since $\gamma_{\ep} \in C^1(W^{1,p}_0(\Omega))$ and for any $n\in \mathbb{N}$, $\gamma_{\ep}(v_n) =0$, it follows that $\gamma_{\ep}(v)=0$. Finally, \eqref{boundedawayfromzero} yields that $\|v_n\|_{W^{1,p}_0(\Omega)} \geq \rho$ for every $n\in \N$. Therefore, $\|v\|_{W^{1,p}_0(\Omega)} \geq \rho>0$ and so $v\in \NN_{\ep}$. This completes the proof.
\end{proof} $\blacksquare$

\begin{lemma}\label{Mnonempty}
Let $f$ satisfy (f1)-(f4). For any $\ep \in \R$ and for any $v \in 
W^{1,p}_0(\Omega)$ such that $v \neq 0$ there exists a unique $\tau_{\ep,v} >0$ such that 
\begin{itemize}
\item[(i)] $\tau_{\ep,v}v \in \NN_{\ep}$,

\item[(ii)] The function $[0,\infty) \ni t \mapsto J_{\ep}(tv)$ is strictly increasing in $[0, \tau_{\ep,v})$ and strictly decreasing in $(\tau_{\ep,v}, \infty)$ and

\item[(iii)] if $DJ_\varepsilon(v)(v)>0$ then $\tau_{\varepsilon,v}>1$ and if $DJ_\varepsilon(v)(v)<0$ then $\tau_{\varepsilon,v}<1$.
\end{itemize}

\end{lemma}

\begin{proof}
Consider the function $g_{v}:[0,\infty) \to \R$ defined by
$$
\begin{aligned}
g_{v}(t):= &J_{\ep}(tv)\\
=& \frac{\ep^2 t^2}{2}\int_{\Omega}|\nabla v|^2 dx + \frac{t^p}{p}\int_{\Omega}|\nabla v|^pdx - \int_{\Omega}F(tv)dx \quad \hbox{for}\quad t\geq 0. 
\end{aligned}
$$

Since $J\in C^2(W^{1,p}_0(\Omega))$, it follows that $g_{v}\in C^2([0,\infty))$ with
\begin{equation}\label{derivative g}
\frac{d g_{v}}{dt}(t)=\ep^2 t\int_{\Omega}|\nabla v|^2 dx + t^{p-1}\int_{\Omega}|\nabla v|^pdx - \int_{\Omega}f(tv)vdx
\end{equation}

and
\begin{equation}\label{second derivative g}
\frac{d^2 g_{v}}{dt^2}(t)=\ep^2\int_{\Omega}|\nabla v|^2 dx + (p-1)t^{p-2}\int_{\Omega}|\nabla v|^pdx - \int_{\Omega}f'(tv)v^2dx
\end{equation}
for any $t\geq 0$.

\medskip
Let $t^*>0$ be such that  $\frac{d g_{v}}{dt}(t^*)=0$. From \eqref{derivative g} and \eqref{second derivative g},
\begin{align*}
\frac{d ^2g_{v}}{d t^2}(t^*)=&\ep^2\int_{\Omega}|\nabla v|^2 dx + (p-1)\left(\int_{\Omega}\frac{f(t^*v)}{t}vdx-\ep^2 \int_{\Omega}|\nabla v|^2 dx\right)\\
& \hspace{8cm} - \int_{\Omega}f'(t^*v)v^2dx\\
=&-(p-2)\ep^2\int_{\Omega}|\nabla v|^2dx+\int_{\Omega}\left( (p-1)\frac{f(t^*v)}{t^*v}-f'(t^*v)\right)v^2dx<0.
\end{align*}

%Let $t>0$ be such that $g_{v}(t) \leq 0$, then 
%$$
%\begin{aligned}
%t\,\frac{d ^2g_{v}}{d t^2}(t)  \leq & \int_{\Omega} \left( -(p-2)\ep^2 t |\nabla v|^2  + (p-1)f(tv)v - f'(tv)tv^2\right)dx\\
%<&  \int_{\{v \neq 0\}} \left((p-1)\frac{f(tv)}{ tv} - f'(tv)\right)tv^2dx\\
%<&0
%\end{aligned}
%$$
Consequently, there exists at most one critical point of $g_{v}$ in $(0, \infty)$, and if such point exists, it must necessarily  be the unique global strict maximum of $g_v$ in $(0,\infty)$. Next, we show the existence of such a critical point. Using \eqref{p-1superlinear},
$$
\begin{aligned}
g_{v}(t) \leq \frac{\ep^2 t^2}{2}\int_{\Omega}|\nabla v|^2 dx + \frac{t^p}{p}\int_{\Omega}|\nabla v|^pdx - a t^m\int_{\Omega}|v|^m dx + b\,{\rm meas}(\Omega)  
\end{aligned}
$$
for any $t\geq 0$.

\medskip 
 
Since $m>p$, taking $t \to \infty$ yields
$$
\lim \limits_{t \to \infty}g_{v}(t)= - \infty,
$$
so that there exists $\tau_{\ep,v} \in [0,\infty)$ such that 
\begin{equation*}%\label{maxpolarcoord}
g_{v}(\tau_{\ep,v})= \max \limits_{t \geq 0} g_{v}(t).
\end{equation*}

\medskip
Now we show that $\tau_{\ep,v}>0$. Observe that $g_v(0)=0$, so it suffices to prove that $\frac{d g_{v}}{dt}(t)>0$ for $t\in (0,\delta)$ with $\delta$ small enough. Indeed from \eqref{estimateimportant1} and \eqref{derivative g} we have for $t>0$
\begin{align} \label{derivative interior}
\begin{split}
\frac{d g_{v}}{dt}(t)&\geq t^{p-1}\left(\int_{\Omega}|\nabla v|^pdx-\int_{\Omega}\frac{f(tv)tv}{t^p}dx\right)\\
&\geq t^{p-1}\left(\int_{\Omega}|\nabla v|^pdx-\mu\int_{\Omega} |v|^pdx-C_\mu t^{q-p}\int_{\Omega}|v|^q dx\right)\\ 
&\geq t^{p-1}\left(\left(1-\frac{\mu}{\lambda_{1,p}}\right)\int_{\Omega}|\nabla v|^pdx-C_\mu t^{q-p}\int_{\Omega}|v|^q dx\right). 
\end{split}
\end{align}

%Integrating \eqref{estimateimportant1} we find that for any $t\in \R$,
%$$
%|F(t)|\leq \frac{\lambda_{1,p}- \mu}{p} |t|^p + \frac{C_{\mu}}{q}|t|^q.
%$$
%
%In consequence, for any $t \in [0, +\infty)$,
%$$
%\begin{aligned}
%g_{v}(t)\geq & \frac{\ep^2 t^2}{2}\int_{\Omega}|\nabla v|^2 dx + \frac{t^p}{p}\int_{\Omega}|\nabla v|^pdx - \frac{\lambda_{1,p}-\mu}{p} t^p\int_{\Omega}|v|^p dx 
%- \frac{C_{\mu}}{q} t^q\int_{\Omega}|v|^q dx\\
%\geq & \frac{\ep^2 t^2}{2}\int_{\Omega}|\nabla v|^2 dx + \frac{ \mu t^p }{p \lambda_{1,p}}\int_{\Omega}|\nabla v|^pdx - \frac{C_{\mu}}{q} t^q\int_{\Omega}|v|^q dx
%\end{aligned}
%$$
%and hence in the case $\ep>0$,
%$$
%\lim \limits_{t\to 0^+}  \frac{g_{v}(t)}{t^2} \geq  \frac{\ep^2}{2}\int_{\Omega}|\nabla v|^2dx, 
%$$
%while the case $\ep=0$ yields
%$$
%\lim \limits_{t\to 0^+}  \frac{g_{v}(t)}{t^p} \geq \frac{\mu}{p \lambda_{1,p}}\int_{\Omega}|\nabla v|^p dx. 
%$$
%

From this it is clear that there exists the required $\delta>0$. Since $v \neq 0$, $g_{v}(\tau_{\ep,v})>0$ and so $\tau_{\ep,v}>0$. Therefore, $\frac{d g_{v}}{dt}(\tau_{\ep,v})=0$, i.e. $DJ_{\ep}(\tau_{\ep,v}v)(\tau_{\ep,v}v)=0$ and hence $\tau_{\ep,v}v \in \NN_{\ep}$ as claimed. 

\medskip
Finally, from the previous discussion, we conclude that $g_{v}$ is strictly increasing in $(0,\tau_{\ep,v})$ and strictly decreasing in $(\tau_{\ep,v},\infty)$. On the other hand, $\frac{dg_v}{dt}(1)=DJ_\varepsilon(v)v$ and $\tau_{\ep,v}>0$ is the only critical point of $g_v$. Thus, if $ \frac{dg_v}{dt}(1)>0$, then $0<1<\tau_{\ep,v})$. If $\frac{dg_v}{dt}(1)<0$, then $1<\tau_{\ep,v}<\infty$ and this concludes the proof. $\blacksquare$
\end{proof}

\begin{lemma}\label{furtherpropofNN}
Let $f$ satisfy hypotheses (f1)-(f4) and let $\ep \in \R$. Then,
\begin{itemize}
\item[(i)] $\mathcal{N}_{\ep}$ is a $C^1$-manifold embedded in $W^{1,p}_0(\Omega)$,
\item[(ii)] the tangent space $T_{v}\NN_{\ep}$ of $\NN_{\ep}$ at $v$ is the set $\{\varphi \in W^{1,p}_0(\Omega)\,:\, D\gamma_{\ep}(v)\varphi =0\}$ and

\item[(iii)] for any $v\in \NN_{\ep}$, $v\notin T_{v}\NN_{\ep}$.

\end{itemize}
\end{lemma}

\begin{proof}
First, we prove {\it (i) and (ii)} by showing that zero is regular value of $\gamma_{\ep}$. Let $v\in \NN_{\ep}$ be arbitrary. Using \eqref{derivgamma} and \eqref{eqnNN},
\begin{align*}
D\gamma_{\ep}(v)(v)&=\int_{\Omega} \left(2 \ep^2|\nabla v|^2+p|\nabla v|^p\right) dx - \int_{\Omega}\left(f'(v)v^2-f(v)v \right) dx \\
&= -(p-2)\ep^2\int_{\Omega} |\nabla v|^2 dx  +  \int_{\Omega}\left( (p-1)f(v)v- f'(v)v^2 \right)dx
\end{align*}
and from hypothesis (f4),
\begin{equation}\label{zeroregvaluegamma}
D\gamma_{\ep}(v)(v)\red{<} \int_{\{v\neq 0\}}\left((p-1)\frac{f(v)}{v}-f'(v)\right)v^2 dx <0.
\end{equation}

Since $W^{1,p}_0(\Omega)= \R v \oplus {\rm Ker} D\gamma_{\ep}(v)$ and $D\gamma_{\ep}(v)|_{\R v}:\R v \to \R $ is bijective, the {\it Implicit Function Theorem} and the fact that $v \in \NN_{\ep}$ is arbitrary yield part {\it (i)}.

\medskip

We also conclude that for any $v\in \NN_{\ep}$,
\begin{equation*}%\label{TangentNN}
T_{v}\NN_{\ep} = {\rm Ker} D\gamma_{\ep}(v)=\{\varphi \in W^{1,p}_0(\Omega)\,:\, D\gamma_{\ep}(v)\varphi =0\}
\end{equation*}
proving {\it (ii)}.  To prove {\it (iii)}, observe from \eqref{zeroregvaluegamma} that $v\notin T_v \mathcal{N}_{\ep}$. This completes the proof of the lemma.   $\blacksquare$

\end{proof}

\medskip
Denote
$$
S_{\infty}:= \left\{v\in W^{1,p}_0(\Omega)\,:\, \|v\|_{W^{1,p}_0(\Omega)}=1\right\}.
$$

\begin{lemma}\label{NNdiffeomSphere}
Let $f$ satisfy (f1)-(f4). For any $\ep \in \R$, the manifold $\NN_{\ep}$ is diffeomorphic to $S_{\infty}$. In particular, $\NN_{\ep}$ is path-connected.
\end{lemma}

\begin{proof}
In this proof we use the same notations as in Lemma \ref{Mnonempty}. Let $\ep \in \R$ be fixed and consider the function $\lambda_{\ep}: S_{\infty} \to \NN_{\ep}$ defined by
$$
\lambda_{\ep}(v):= \tau_{\ep,v}v \quad \hbox{for} \quad v\in S_{\infty},
$$ 
where $\tau_{\ep,v}>0$ is the unique positive value such that $\tau_{\ep,v}v \in \NN_{\ep}$ as described in Lemma \ref{Mnonempty}. 

\medskip
It is direct to verify that $\lambda_{\ep} : S_{\infty} \to \NN_{\ep}$ is injective. On the other hand, given ${\rm v}\in \NN_{\ep}$, $\tau_{\ep,{\rm v}}=1$. Set
$$
v:=\frac{{\rm v}}{\|{\rm v}\|_{W^{1,p}_{0}(\Omega)}} \in S_{\infty}
$$
so that $\tau_{\ep,{v}}= \|{\rm v}\|_{W^{1,p}_0(\Omega)}$ and $\lambda_{\ep}\left({v}\right) = \tau_{\ep,{v}} {v} = {\rm v}$. Therefore, $\lambda_{\ep}$ is bijective.

\medskip
Next, we prove that $\lambda_{\ep}\in C^1(S_{\infty})$. Consider the function 
$$
\xi:\R \times (0,\infty)\times \left(W^{1,p}_0(\Omega)\setminus \{0\}\right)\to \R, \qquad \xi(\varepsilon, t,v):=\gamma_{\ep}(tv).
$$ 

\medskip
Observe that for any $\ep \in \R$, $t>0$ and any $v\in W^{1,p}_0(\Omega)$ with $v\neq 0$, 
$$
\xi(\varepsilon,t,v)=0 \quad \Longleftrightarrow \quad tv\in \NN_{\ep} \quad \Longleftrightarrow \quad \lambda_{\ep}(v)=tv.
$$

Also, $\xi$ is a $C^1-$function and from \eqref{derivgamma} for any $\ep \in \R$, any $t>0$ and $v\in W^{1,p}_0(\Omega)$ with $v\neq 0$ and $\xi(\varepsilon, t,v)=0$,
\begin{align*}
\frac{\partial \xi}{\partial t}(\varepsilon, t,v)&= D\gamma_{\ep}(tv)v\\
&=\int_{\Omega} \left(2 \ep^2 t|\nabla v|^2+pt^{p-1}|\nabla v|^p\right) dx - \int_{\Omega}\left(f'(tv)t v^2-f(tv)v \right) dx \\
&= -(p-2)\ep^2 t\int_{\Omega} |\nabla v|^2 dx  +  \int_{\Omega}\left( (p-1)f(t v)v- f'(tv)t v^2 \right)dx.
\end{align*}

From hypothesis (f4),
\begin{equation*}
\frac{\partial \xi}{\partial t}(\varepsilon, t,v)< \int_{\{v\neq 0\}}\left((p-1)\frac{f(tv)}{tv}-f'(tv)\right)t v^2 dx <0.
\end{equation*}

Thus, the {\it Implicit Function Theorem} yields that the mapping 
\begin{equation*}%\label{radialvariable}
(\varepsilon,v) \mapsto \tau_{\ep,v}>0
\end{equation*}
belongs to $C^1\left(\R \times \left( W^{1,p}_0(\Omega)\setminus\{0\}\right)\right)$ and we conclude, in particular, that $\lambda_{\ep} \in C^1(S_{\infty})$. Finally, from Lemma \ref{QualitLemma}, ${\rm Img}(\lambda_{\ep})\subset W^{1,p}_0(\Omega)- B_{\rho}(0)$ and since $(\lambda_\varepsilon)^{-1}(\rm v)=\frac{{\rm v}}{\|{\rm v}\|_{W^{1,p}_{0}(\Omega)}}$, we conclude that $\lambda_{\ep}$ is a diffeomorphism. $\blacksquare$
\end{proof}

{\bf Remark.} An important corollary of the previous proof, is the $C^1-$ dependence on $\ep \in \R$ of $\tau_{\varepsilon,v}$.\\

%\begin{lemma}\label{criticalpointsinNN}
%Let $\varepsilon \geq 0$. A point $v\in \mathcal{N}_{\ep}$ is a critical point of $J_{\ep}$ if and only if $v$ is a critical point of $J_{\ep}|_{\NN_{\ep}}$.
%\end{lemma}
%
%
%\begin{proof}
%Let $v \in \NN_{\ep}$. From Lemma \ref{NNdiffeomSphere}, $T_v\mathcal{N}_{\ep}$ has codimension 1 and $v\notin T_v\mathcal{N}_{\ep}$. Therefore, 
%\begin{equation}\label{geometricdecomp}
%W^{1,p}_0(\Omega)=\R v \oplus T_{v}\NN_{\ep}.
%\end{equation}
%
%If $v\in \NN_{\ep}$ is a critical point of $J_{\ep}$, then from \eqref{geometricdecomp} $DJ_{\ep}(v)=0$ on $T_{v}\NN_{\ep}$ and consequently $v$ is critical point of $J_{\ep}|_{\NN_{\ep}}$. 
%
%
%\medskip
%On the other hand if $v\in \NN_{\ep}$ is a critical point of $J_{\ep}|_{\NN_{\ep}}$, then $DJ_{\ep}(v)=0$ on $T_{v}\NN_{\ep}$ and the fact that $v\in \NN_{\ep}$ together with \eqref{geometricdecomp} imply that $v$ is a critical point of $J_{\ep}$ in $W^{1,p}_0(\Omega)$.  $\blacksquare$
%\end{proof} 

Now, we study the set $\MM_{\ep}$ defined in \eqref{Hemispheres} and the minimization problem \eqref{leastnodalenergyslncondition}. Observe that 
$$
\begin{aligned}
\MM_{\ep}=& \{v\in W^{1,p}_0(\Omega)\,:\, v^+,v^- \neq 0, \quad \gamma_{\ep}(v^+) = \gamma_{\ep}(v^-)=0\}\\
=& \{v\in W^{1,p}_0(\Omega)\,:\, v^+,v^- \neq 0, \quad DJ_{\ep}(v)v^+ = DJ_{\ep}(v)v^- =0\},
\end{aligned}
$$
where we recall that for any $v\in W^{1,p}_0(\Omega)$,
$$
J_{\ep}(v)= \int_{\Omega}\left(\frac{\ep^2}{2}|\nabla v|^2 + \frac{1}{p}|\nabla v|^p - F(v)\right)dx.
$$

\begin{lemma}\label{alpha=beta}
Let $f$ satisfy (f1)-(f4). For any $\ep \in \R$, the numbers
$$
\alpha_{\ep}:= \inf \limits_{v\in \MM_{\ep}} J_{\ep}(v)
\quad \hbox{and}
\quad \beta_{\ep}:= \inf \limits_{v\in W^{1,p}_0(\Omega), \,\,
	v^+ ,\, v^-\neq 0} \left(\max \limits_{t,\,s > 0} J_{\ep}(tv^+ + sv^-)\right),
$$
are well defined and $\alpha_{\ep} = \beta_{\ep}$.
\end{lemma}

\begin{proof}
Let $\ep \in \R$ be arbitrary, but fixed and let $v\in W^{1,p}_0(\Omega)$ be also arbitrary and such that $v^+,v^- \neq 0$. Consider the function $h_{\ep}:[0,\infty)\times [0,\infty) \to \R$ defined by
$$
h_{\ep}(t,s):= J_{\ep}(t v^+ + sv^-) \quad \hbox{for}\quad t,s \geq 0.
$$

Observe that
$$
\begin{aligned}
h_{\ep}(t,s) =&\int_{\Omega} \left( \frac{\ep^2 t^2}{2} |\nabla v^+|^2 + \frac{t^p}{p}|\nabla v^+|^p -  F(t v^+) \right) dx \\
& + \int_{\Omega} \left( \frac{\ep^2 s^2}{2} |\nabla v^-|^2 + \frac{s^p}{p}|\nabla v^-|^p -  F(s v^-) \right) dx
\end{aligned}
$$ 
so that $h_{\ep}(t,s)=J_{\ep}(tv^+) + J_{\ep}(sv^-)$ for any $t,s \geq 0$. Using the the notation of the proof of Lemma \ref{Mnonempty} we get the decomposition $h_{\ep}(t,s)=g_{v^+}(t)+g_{v^-}(s)$. Also, since $J_{\ep}\in C^2(W^{1,p}_0(\Omega))$, $h_{\ep}\in C^2\left([0,\infty)\times[0,\infty)\right)$ and
\begin{equation*}%\label{gradienthep}
\begin{aligned}
\nabla h_{\ep}(t,s)
%=& \left(DJ_{\ep}(tv^+ + sv^-)v^+, %DJ_{\ep}(tv^+ + sv^-)v^-\right)\\
=&\left(DJ_{\ep}(tv^+)v^+, DJ_{\ep}(sv^-)v^-\right)=\left(\frac{d g_{v^+}}{dt}(t),\frac{d g_{v^-}}{ds}(s)\right).
\end{aligned}
\end{equation*}

Hence, since $g_{v^+}$ and $g_{v^-}$ have only one critical point each, namely $t_{\ep,v}>0$ and $s_{\ep,v}>0$, respectively, we have that this pair provides the unique critical point of $h$. On the other hand, we have that $h_{\ep}(t,s)= g_{v^+}(t)+g_{v^-}(s)\to - \infty $, as $|(t,s)| \to + \infty$. Also, $h_{\ep}(0,0)=0$ and by \eqref{derivative interior}, we find $\delta>0$ small and such that if $s\geq 0$ and $t \in (0,\delta)$ then $\partial_{t}h_{\ep}(t,s)>0$. Similarly, if $t\geq 0$ and  $s\in (0,\delta)$ then $\partial_{s}h_{\ep}(t,s)>0$. 

\medskip
We conclude that for any $v\in W^{1,p}_0(\Omega)$, with $v^+,v^- \neq 0$, there exists a unique pair $(t_{\ep,v},s_{\ep,v}) \in (0,\infty) \times (0,\infty)$ such that
$$
J_{\ep}(t_{\ep,v} v^+ + s_{\ep,v}v^-) = \max \limits_{t,s \geq 0} J_{\ep}(tv^+ + sv^-)=\max \limits_{t,s > 0} J_{\ep}(tv^+ + sv^-).
$$

In particular, $t_{\ep,v} v^+ + s_{\ep,v} v^- \in \MM_{\ep}$. From this fact it follows that $\MM_{\ep}$ is non-empty and combining it with \eqref{coercivity} it follows that $\alpha_{\ep}$ and $\beta_{\ep}$ are well defined.

\medskip
Moreover,
$$
\begin{aligned}
\alpha_{\ep} \leq & J_{\ep}(t_{\ep,v} v^+ + s_{\ep,v} v^-)\\
=& \max \limits_{t,s \geq 0} J_{\ep}(tv^+ + sv^-).
\end{aligned}
$$

Since $v$ is arbitrary, we conclude that $\alpha_{\ep}\leq \beta_{\ep}$.

\medskip
Next, we prove the reverse inequality. Let $v\in \MM_{\ep}$ be arbitrary. Since $v^+,v^- \in \NN_{\ep}$, the proof of Lemma \ref{Mnonempty} yields that $J_\varepsilon(v^+)=\max \limits_{t\geq 0} J_{\ep}(tv^+)$ and $J_\varepsilon(v^-)=\max \limits_{s\geq 0} J_{\ep}(sv^-)$, i.e. $\tau_{\ep,v^+}=\tau_{\ep,v^-} =1$.

\medskip
Let $t_{\ep,v},s_{\ep,v}>0$ be such that $J_{\ep}(t_{\ep,v}v^+ + s_{\ep,v}v^-)= \max \limits_{t,s\geq 0}J_{\ep}(tv^+ + sv^-)$. Then, 
$$
\max \limits_{t,s \geq 0} J_{\ep}(tv^+ + sv^-)=\max \limits_{t,s \geq 0} \{J_{\ep}(tv^+) +J_\varepsilon (sv^-)\}=J_\varepsilon(v^+)+J_\varepsilon(v^-)=J_\varepsilon(v),
$$
implying that $(t_{\ep,v},s_{\ep,v})=(1,1)$.

\medskip
Therefore, 
$$
\begin{aligned}
\beta_{\ep} \leq & \max \limits_{t,s \geq 0} J_{\ep}(tv^+ + sv^-)\\
=& J_{\ep}(v). 
\end{aligned}
$$

Since $v\in \MM_{\ep}$ is arbitrary, we find that $\beta_{\ep} \leq \alpha_{\ep}$. This concludes the proof.  $ \blacksquare$
\end{proof}

\medskip

\medskip
As pointed out in \cite{BartschWeth2003}, due to the lack of regularity of the functions 
$$
W^{1,p}_0(\Omega)\ni v \mapsto v^{\pm} \in W^{1,p}_0(\Omega),
$$
(even in the case $p=2$), the set $\MM_{\ep}$ is not a submanifold of $W^{1,p}_0(\Omega)$ and hence it is not automatically clear that $u \in \MM_{\ep}$ solving \eqref{leastnodalenergyslncondition} is a critical point of $J_{\ep}$. For this reason we use a version of the deformation lemma (see Lemma 2.13 in \cite{WILLEMBOOK}) to show that any local minimizer of $J_{\ep}|_{\MM_{\ep}}$ is a critical point of $J_{\ep}$.

\medskip
%\blue{Please review carefully this proof. I fixed some details that were not incorrect, but incomplete nontheless.}

\begin{lemma}\label{Deformation}
Let $f$ satisfy (f1)-(f4) and let $\ep \in \R$. Assume $u\in \MM_{\ep}\cap V$, where $V\subset W^{1,p}_0(\Omega)$ is an open set.  If 
$$
J_{\ep}(u)= \min \limits_{v\in \MM_{\ep}\cap V}J_\varepsilon(v),
$$
then $u$ is a critical point of $J_\varepsilon$ in $W^{1,p}_0(\Omega)$.
\end{lemma}

\begin{proof}
Arguing by contradiction, assume that $u$ is not a critical point in $W^{1,p}_0(\Omega)$ of $J_\varepsilon$ so that 
$$
2\eta_{\ep}:=\Vert DJ_{\ep}(u)\Vert_{(W_0^{1,p}(\Omega))^*}>0.
$$ 
	
Choose $r_{\ep}>0$ such that $B_{r_{\ep}}(u)\subset V$ and for every $v\in B_{r_{\ep}}(u)$,
$$
	\Vert DJ_{\ep}(v)\Vert_{(W_0^{1,p}(\Omega))^*}>\eta_{\ep} >0. 
	$$
	
Fix $a,b \in \R$ with $a<1<b$. Set $D:=(a,b)\times(a,b)\subset \R^2$ and assume further that $a,b$ are such that $\{tu^+ + su^- \,:\,(t,s)\in  \overline{D}\}\subset V$.

\medskip
Consider the function $h_{\ep}(t,s):=J_{\ep}(tu^+ + s u^-)$ for $(t,s) \in [0,\infty) \times [0,\infty)$, introduced in the proof of Lemma \ref{alpha=beta} and recall that for any $(t,s) \in \R^2-\{(1,1)\}$,
	\begin{equation}\label{inequality h}
	h_{\ep}(t,s)=J_{\ep}(tu^+) + J_{\ep}(su^-)< J_{\ep}(u^+) + J_{\ep}(u^-) =J_{\ep}(u) = h_{\ep}(1,1).
	\end{equation}

In particular  
	\begin{equation*}%\label{ineqI}
	\max \limits_{(t,s)\in \partial D} h_{\ep}(t,s) < h_{\ep}(1,1).
	\end{equation*}
	
	\medskip
Let $\delta\in \left(0,\frac{r_\varepsilon}{3}\right)$ be fixed and small enough so that $B_{3\delta}(u) \subset V$ and
	\begin{equation}\label{delta deformation}
	 \{tu^+ + su^- \,:\,(t,s)\in \partial D\}\subset V\setminus B_{3\delta}(u).
	\end{equation}
		
Set $S:=\overline{B}_{\delta}(u)$, $c:=J_\varepsilon (u)$ and fix $\theta_\varepsilon\in \left(0,\frac{\delta \eta_\varepsilon}{8}\right)$. Using Lemma 2.3 in \cite{WILLEMBOOK} we find a continuous deformation $\Lambda_{\ep}:[0,1]\times W^{1,p}_0(\Omega) \to W^{1,p}_0(\Omega)$ such that
	\begin{itemize}
		\item[(a)] $\Lambda_{\ep}(1,v) =v$ for $v\notin J_{\ep}^{-1}\left([c-2\theta_{\ep},c+2\theta_{\ep}]\right)\cap B_{3\delta}(u)$,
		
		\item[(b)] $\Lambda_{\ep}(1, J_{\ep}^{c + \theta_{\ep}} \cap \overline{B}_{\delta}(u)) \subset J_{\ep}^{c - \theta_{\ep}}$,
		
		\item[(c)] $J_{\ep}(\Lambda_{\ep}(1,v)) \leq J_{\ep}(v)$ for every $v\in W^{1,p}_0(\Omega)$,
		\item[(d)] $\Lambda_{\ep}(t,\cdot)$ is an homeomorphism of $W_0^{1,p}(\Omega)$ for every $t\in [0,1]$.	
%		\red{\item[(e)] $\|\Lambda_{\ep}(\cdot\,,v) -v\|_{W^{1,p}_0(\Omega)}\leq \delta$ for any $v\in W^{1,p}_0(\Omega)$.}
	\end{itemize}  

From {\rm (c)} and \eqref{inequality h} for every $(t,s)\in \overline{D}-\{(1,1)\}$,
$$
J_{\ep}(\Lambda_{\ep}(1, tu^+ + s u^-)) \leq J_{\ep}(tu^+ + s u^-) < J_{\ep}(u).
$$

On the other hand, item {\rm(c)} also implies that $J_{\ep}(\Lambda_{\ep}(1,u))\leq J_{\ep}(u)$ and since $u\in J_{\ep}^{c+ \theta_{\ep}}\cap B_{\delta}(u)$ then item {\rm (b)} yields
$$
J_{\ep}(\Lambda_{\ep}(1,u)) \leq c- \theta_{\ep} <c = J_{\ep}(u).
$$

Therefore,
\begin{equation}\label{contradiction1}
\max \limits_{(t,s)\in \overline{D}} J_{\ep}(\Lambda_{\ep}(1, tu^+ + s u^-)) < J_{\ep}(u).
\end{equation}

Consider the mapping $\red{\sigma}_{\ep}(t,s):= \Lambda_{\ep}(1, tu^+ + su^-)$ for $(t,s)\in \overline{D}$. Next, we claim that $\red{\sigma}_{\ep}(D) \cap \MM_{\ep}\cap V$ is non-empty. If this is the case, then 
$$
J_{\ep}(u) \leq \min \limits_{v\in \sigma_{\ep}(D) \cap \MM_{\ep}\cap V}J_{\ep}(v)
$$
and hence contradicting \eqref{contradiction1}.

	\medskip
To prove the claim we proceed as follows. Consider the functions
	$$
	\begin{aligned}
	\Psi_0(t,s):=& \left(DJ_{\ep}(tu^+)u^+, DJ_{\ep}(su^-)u^-\right),\\
	\Psi_1(t,s):=&\left(\frac{1}{t}DJ_{\ep}(\red{\sigma}_{\ep}^+(t,s))\red{\sigma}_{\ep}^+(t,s),\frac{1}{s}DJ_{\ep}(\red{\sigma}_{\ep}^-(t,s))\red{\sigma}_{\ep}^-(t,s)\right)
	\end{aligned}
	$$
	for $(t,s)\in \overline{D}$.
	
	\medskip
	Since $u\in \MM_{\ep}$, $(1,1)\in D$ is the unique point of maximum of $h_{\ep}$. We conclude that
	$$
	\deg(\Psi_0,D,0)=1.
	$$
	
	From \eqref{delta deformation} and item {\rm (a)}, $\red{\sigma}_{\ep}(t,s)= tu^+ + su^-$ for $(t,s)\in \partial D$, so that $\Psi_0=\Psi_1$ on $\partial D$ and consequently,
	$$
	{\deg}(\Psi_1,D,0)={\rm deg}(\Psi_0,D,0)=1.
	$$ 
	
Therefore, for some $(t_1,s_1) \in D$, $\Psi_1(t_1,s_1)=(0,0)$, i.e. $\red{\sigma}_{\ep}(t_1,s_1)\in \MM_{\ep}$. Finally, items {\rm(a)} and {\rm (d)} imply that $\Lambda_{\ep}(1,\cdot)$ is a homeomorphism with $\Lambda_{\ep}(1, \cdot)|_{W^{1,p}_0(\Omega)-B_{3\delta}(u)}= Id|_{W^{1,p}_0(\Omega)-B_{3\delta}(u)}$. Consequently, $\Lambda_{\ep}(1,V) \subset V$ and since $(t_1,s_1)\in D$, then $\Lambda_{\ep}(1,t_1u^+ + s_1 u^-) \in V$, i.e $\sigma_{\ep}(t_1,s_1)\in V$. This concludes the proof of the claim and also the proof of the lemma. $\blacksquare$
\end{proof}

%\blue{The statement of this lemma has been rewritten. Please review it careful so that it says what it means to say.}
\begin{lemma}\label{Nehari compactness}
Let $f$ satisfy (f1)-(f4) and let $\ep \in \R$.  Assume that $B\subset W^{1,p}_0(\Omega)$ is either an open ball or $B=W^{1,p}_0(\Omega)$. Let $\{v_n\}_{n\in \N}\subset \MM_{\ep}\cap B$ and $v\in W^{1,p}_0(\Omega)$ be such that $v_{n}\rightharpoonup v$ weakly in $W_0^{1,p}(\Omega)$ and let $\tau_{\varepsilon,v^+},\tau_{\varepsilon,v^-}>0$ be such that $\tau_{\varepsilon,v^+}v^+ + \tau_{\varepsilon,v^-}v^- \in \MM_{\ep}\cap B$. Then the following holds true. 
\begin{itemize}
\item[(i)] $v^+\neq 0$ and $v^-\neq 0$. 
\item[(ii)] If $\{v_n\}_{n\in \N}$ is a minimizing sequence for $J_\varepsilon|_{\MM_\varepsilon\cap B}$, i.e. $J_{\ep}(v_n) \to \min \limits_{w\in \MM_{\ep}\cap B}J_{\ep}(w)$, then $v_{n} \to v$ strongly in $W^{1,p}_0(\Omega)$, $v\in \MM_{\ep}\cap B$ and $J_{\ep}(v)= \min \limits_{w\in \MM_{\ep}\cap B} J_{\ep}(w)$.
\end{itemize}
In particular, if $J_{\ep}(v_n) \to \min \limits_{w\in \MM_{\ep}}J_{\ep}(w)$, then $v\in \MM_{\ep}$ and  $J_{\ep}(v)= \min \limits_{w\in \MM_{\ep}
} J_{\ep}(w)$.
\end{lemma}

%\blue{Some details have been modified in this proof. Please read it carefully to double check for errors.}

\begin{proof}
We proceed as in section 3 of \cite{CASTROCOSSIONEUBERGER1997}. First, notice that 
\begin{equation}\label{subsequence1}
v^+_n \rightharpoonup v^+, \quad v_n^- \rightharpoonup v^- \quad \hbox{weakly in}\quad W^{1,p}_0(\Omega)   
\end{equation}
and
\begin{equation}\label{subsequence3}
\|v^{+} \| _{W^{1,p}_0(\Omega)} \leq \lim\limits_{n\rightarrow \infty} \|v^{+}_n \| _{W^{1,p}_0(\Omega)} \ \text{ and } \ \|v^{-} \| _{W^{1,p}_0(\Omega)} \leq \lim\limits_{n\rightarrow \infty} \|v^{-}_n \| _{W^{1,p}_0(\Omega)} .
\end{equation}

The compactness of the Sobolev embeddings and \eqref{subsequence1} imply that, up to a subsequence,
\begin{equation*}%\label{subsequence2}
v^+_n \to v^+, \quad v_n^- \to v^- \quad \hbox{strongly in}\quad L^r(\Omega)   
\end{equation*}
for any $r\in [1,p^*)$. Also, $v^+ \geq 0$ and $v^- \leq 0$ a.e. in $\Omega$. By taking further subsequences, if necessary, we may also assume that the sequences $\{ \|v^{+}_n \| \} _n$ and $\{ \|v^{-}_n \| \} _n$ converge in $\Bbb R$.

\medskip
Vainberg's Lemma (see \cite{VAINBERG64}) yields
\begin{equation*}%\label{vainberg1}
\int_{\Omega}F(v^{\pm}_n)\to \int_{\Omega} F(v^{\pm}),
\end{equation*}
\begin{equation*}%\label{vainberg2}
\int_{\Omega}f(v_n^\pm)v_n^\pm\to \int_{\Omega} f(v^\pm)v^\pm,
\end{equation*}
as $n\to \infty$. Since $v_n\in \MM_{\ep}$, 
\begin{equation*}%\label{inMM}
DJ_{\ep}(v_n^{\pm})v^{\pm}_{n} =0 \quad \hbox{for every} \quad n\in \mathbb{N}
\end{equation*}
and using Lemma \ref{QualitLemma} we find that
\begin{equation*}
\begin{aligned}
\int_{\Omega}f(v^{\pm})v^{\pm} dx=&\lim \limits_{n\to \infty}\int_{\Omega}f(v_n^{\pm})v_n^{\pm}dx \\ 
\geq& \lim \limits_{n\to \infty}\int_{\Omega}|\nabla v_n^{\pm}|^p dx\\\geq&  \rho^p,
\end{aligned}
\end{equation*}
so that $v^{\pm} \neq 0$ in $W^{1,p}_0(\Omega)$. This proves {\it (i)}.\\

Next, we prove {\it (ii)}. First we prove that $v_n \to v$ strongly in $W^{1,p}_0(\Omega)$. We proceed by showing that in \eqref{subsequence3} both equalities hold.

\medskip
Let us argue by contradiction. Assume either 
\begin{equation}\label{subsequence4}
\|v^{+} \| _{W^{1,p}_0(\Omega)} < \lim\limits_{n\rightarrow \infty} \|v^{+}_n \| _{W^{1,p}_0(\Omega)} \ \text{ or } \ \|v^{-} \| _{W^{1,p}_0(\Omega)} < \lim\limits_{n\rightarrow \infty} \|v^{-}_n \| _{W^{1,p}_0(\Omega)} .
\end{equation}
Set $a:=\tau_{\varepsilon,v^+}$ and $b:=\tau_{\varepsilon,v^-}$. Since $av^+ + bv^- \in \MM_{\ep}\cap B$,
$$
\begin{aligned}
\inf \limits_{w\in \MM_{\ep}\cap B}J_{\ep}(w) 
\leq & J_{\ep}(av^+ + bv^-)=J_{\ep}(av^+ ) + J_{\ep} (bv^-)\\
< & \lim_{n\to \infty} J_{\ep}( a v_n^+)+\lim \limits_{n\to \infty}J_{\ep}( bv_n^-) \ \ \text{ (from }  \eqref{subsequence4}) \\
\leq & \lim_{n\to \infty} J_{\ep}(  v_n^+)+\lim \limits_{n\to \infty}J_{\ep}( v_n^-)  \ \ \text{ (since } v_n^+ , v_n^- \in \NN_{\ep}) \\
=& \lim_{n\to \infty} J_{\ep}(  v_n^+ + v_n^-) =\inf \limits_{w\in \MM_{\ep}\cap B}J_{\ep}(w),
\end{aligned}
$$
which is a contradiction.  Thus, the equalities hold in \eqref{subsequence3}. Since $W_0^{1,p}(\Omega)$ is uniformly convex (see Theorem 2.6 in \cite{AdamsFournier}) the convergence of $\{v_n\}_{n\in \N}$ to $v$ in $W^{1,p}_0(\Omega)$ is strong. 

\medskip
Since $v_n \to v$ strongly in $W^{1,p}_0(\Omega)$ and $J_{\ep}$ is a $C^1-$ function, we conclude that $v^+,v^-\in \NN_{\ep}$, i.e. $v\in \MM_{\ep}$ or equivalently $t_{\ep,v}=s_{\ep,v}=1$. From our assumptions, this in turn implies that $v\in B$.

%Since $u_n^+, u_n^- \in \NN_{\ep}$ and $t_{\ep,u_n^+}=t_{\ep,u^-_n}=1$ satisfy \eqref{maxpolarcoord}, either 
%$$
%J_{\ep}(a u_n^+) \leq J_{\ep}(u_n^+) \quad \hbox{and} \quad J_{\ep}(b u_n^-)\leq J_{\ep}(u_n^-) 
%$$
%with one of the inequalities being strict. Thus,
%$$
%\min \limits_{v\in \MM_{\ep}}J_{\ep}(v)  \leq \liminf_{n\to \infty} J_{\ep}( a u_n^+)+\liminf \limits_{n\to \infty}J_{\ep}( bu_n^-)< \liminf_{n\to \infty} J_{\ep}( u_n^+ + u_n^-) = \min \limits_{v\in \MM_{\ep}}J_{\ep}(v) 
%$$

\medskip
Finally, If $\{v_n\}_{n\in \N}$ is a minimizing sequence for $J_\varepsilon|_{\MM_\varepsilon\cap B}$, then $J_{\ep}(v)=\min \limits_{w\in \MM_{\ep}\cap B}J_{\ep}(v)$. This completes the proof of the lemma. $\blacksquare$
\end{proof}   

\begin{proof}{\it Proof of Theorem \ref{Theorem1}.} 
Let $\{v_n\}_{n\in \N} \subset \MM_\varepsilon$  be a minimizing sequence, i.e. $J_\varepsilon(v_n)\to \alpha_\varepsilon$. From Lemma \ref{QualitLemma}, $J_\varepsilon|_{\NN_{\ep}}$ is coercive. Consequently, $\{v_n\}_{n\in \N}$ is bounded. From the reflexi\-vity of $W^{1,p}_0(\Omega)$, there exists $u\in W_0^{1,p}(\Omega)$ such that, up to a subsequence $v_n \rightharpoonup u$ weakly in $W_0^{1,p}(\Omega)$. A direct application of Lemma \ref{Nehari compactness} implies that  $u\in \MM_{\ep}$ solves \eqref{leastnodalenergyslncondition}. Finally,  Lemma \ref{Deformation} implies that $u$ is a least energy nodal solution of \eqref{eq mean problem} and Lemma \ref{alpha=beta} yields \eqref{nodalcharacteerization}. This completes the proof of the theorem. $\blacksquare$
\end{proof} 

\medskip

We conclude this section providing a result that relates the number of nodal regions of a solution with its Morse index.
\begin{lemma}\label{nodal regions}
Let $f$ satisfy (f1) and (f4) and let $\ep \in \R$. Assume that $u\in W_0^{1,p}(\Omega)$ is a weak solution of \eqref{eq mean problem} with $N_{\ep}(u)$ nodal regions. Then $N_{\ep}(u)\leq m_\varepsilon(u)$.
\end{lemma}

%\blue{The reference for the nodal part belonging to the Sobolev space is included.}

\begin{proof}
Let $C$ be a nodal region of $u$ and define $v:=\mathds{1}_{C}u$. Lemma 1 in \cite{MULLERPFEIFFER1985} yields that $v\in W_0^{1,p}(\Omega)$ and
$$
\nabla v(x)=\left\{
\begin{array}{ccc}
0,&\quad \hbox{a.e.} &  x\in \Omega \setminus C\\
\nabla u(x), &\quad  \hbox{a.e.} & x\in C.  
\end{array}
\right.
$$

Thus, if $D$ is another nodal region of $u$ and $w:=\mathds{1}_{D}u$, then from \eqref{2ndDerivJep} we have that
	\begin{equation*}
	D^2J_\varepsilon(u)(v,w)=0.
	\end{equation*}

We conclude that given any $t,s \in \R$,
	\begin{equation*}
	D^2J_\varepsilon(u)(tv+sw,tv+sw)=t^2D^2J_\varepsilon(u)(v,v)+s^2D^2J_\varepsilon(u)(w,w).
	\end{equation*}

Therefore, to prove the lemma, it suffices to show that $D^2J_\varepsilon(u)$ is negative definite along any direction $v$ of the form $\mathds{1}_{C}u$ with $C$ an arbitrary nodal region of $u$.

\medskip
Notice that since $u$ is critical point, using the definition of $v$
	$$DJ_\varepsilon(u)(v)=\int_{\Omega} \left(\ep^2 + |\nabla v|^{p-2}\right)\nabla v\cdot \nabla v dx - \int_{\Omega} f(v)v dx=0.$$
	Using the latter equation and (f4) we obtain (see also the proof of Lemma \ref{Mnonempty})
	\begin{align*}
	D^2J_\varepsilon(u)(v,v)&= \int_{\Omega} \left(\ep^2 + |\nabla v|^{p-2}\right)\nabla v\cdot \nabla v dx 
	+ \int_{\Omega}(p-2)|\nabla v|^{p} dx - \int_{\Omega} f'(v)v^2 dx\\
	&=(2-p)\int_{\Omega} \ep^2|\nabla v|^2
	+ \int_{\{v\neq 0\}}\left((p-2)\frac{f(v)}{v} - f'(v)\right)v^2 dx<0.
	\end{align*}
	Proceeding inductively on the number of nodal regions the result follows.   $\blacksquare$
\end{proof}
% Using the weak lower semicontinuity of the norms $\|\cdot\|_{W^{1,2}_0(\Omega)}$ and $\|\cdot\|_{W^{1,p}_0(\Omega)}$ and taking $n \to \infty$ in \eqref{inMM},
%$$
%DJ_{\ep}(u^{\pm})u^{\pm} \leq 0.
%$$
%
%
%Let us assume by contradiction that
%\begin{equation}\label{contradicexistence}
%DJ_{\ep}(u^{+})u^{+} < 0 \quad \hbox{or} \quad DJ_{\ep}(u^{-})u^{-} <0.
%\end{equation}
%
%Set $a:=t_{\ep,u^+}$ and $b:=t_{\ep,u^-}$ the unique positive numbers such that $a u^+, bu^- \in \NN_{\ep}$. We claim that $DJ_{\ep}(u^+)u^+<0$ implies that either $a\in (0,1)$.
%
%\medskip
%Indeed, assume that $DJ_{\ep}(u^+)u^+ <0$ and hence $a\neq 1$. If $a > 1$, then using (f4) and that $a u^+ \in \NN_{\ep}$,
%$$
%\begin{aligned}
%\int_{\Omega}\left(\frac{\ep^2}{a^{p-2}} |\nabla u^+|^2 + |\nabla u^+|^p\right)dx =& \frac{1}{a^{p-1}}\int_{\Omega}f(a u^+)au^+dx\\ 
%=&\int_{\{u^+ >0\}} \frac{f(a u^+)}{(au^+)^{p-1}} a(u^+)^p dx\\
%\geq & \int_{\{u^+ >0\}} \frac{f(u^+)}{(u^+)^{p-1}} a(u^+)^p dx \\
%=& a\int_{\Omega}f(u^+)u^+dx \\
%>&a \int_{\Omega}\left(\ep^2|\nabla u^+|^2 + |\nabla u^+|^p\right)dx
%\end{aligned}
%$$
%and we conclude that $1> a^{p-1}$, i.e. $a<1$, which contradicts the assumption $a>1$. Hence, $a\in (0,1)$. Similarly, we find that in the case $DJ_{\ep}(u^-)u^- <0$, then $b\in (0,1)$

%\medskip
%Therefore, from \eqref{contradicexistence}, either $a\in (0,1)$ or $b\in (0,1)$. 

%%%%%%%%%%%%%%%%%%%%%%%%%%%
\section{Regularity}
In this part we obtain regularity results for solutions of \eqref{eq mean problem}. This results will play a crucial role when computing the Morse index of least energy nodal solutions in Section \ref{computMorseIndexNondeg}. The first result states the uniform boundedness of $u$. 

\begin{lemma}\label{lemmaregul1}
Let $\ep \in \R$ and $f$ satisfy hypothesis (f1). Any weak solution $u\in W^{1,p}_0(\Omega)$ of \eqref{eq mean problem}, in the sense of the Definition \ref{solution}, belongs to $L^{\infty}(\Omega)$.
\end{lemma}

\begin{proof}
The reader is referred to Lemma 3.1 in  \cite{CingolaniVannella2003}, but for the sake of completeness and clarity we present a brief sketch of it. 

\medskip
Let $u \in W^{1,p}_0(\Omega)$ satisfy the integral identity \eqref{eq weak mean problem}. If $p>N$, the Sobolev embedding $W^{1,p}_0(\Omega) \hookrightarrow C^{0,1-\frac{N}{p}}(\overline{\Omega})$ yields the conclusion. 

\medskip
Assume now that $p\in (1,N]$ and that $q \in (p,p^*)$. Let $j\in \mathbb{N}$ be arbitrary, but fixed and consider the function
$$
\chi_j(\zeta):= \left\{
\begin{aligned}
\zeta + j \quad &\hbox{for} \quad \zeta < -j,\\
0 \quad &\hbox{for} \quad |\zeta| \leq j,\\
\zeta - j \quad &\hbox{for} \quad \zeta > j.
\end{aligned}
\right.
$$

Since $\chi_j \in W^{1,\infty}_{loc}(\overline{\R})$, $\chi'_j \in L^{\infty}(\R)$ and $\chi_j(0)=0$,
$$
|\chi_j(u)| \leq |u| 
\qquad \hbox{and} 
\qquad \nabla \chi_j(u):= \mathds{1}_{\Omega_j}\nabla u \quad \hbox{a.e. in} \quad \Omega,
$$ 
where $\Omega_j:=\{|u|>j\}$ and $\mathds{1}_{\Omega_j}$ is the characteristic function of $\Omega_j$.

\medskip 
Proceeding in the same fashion as in the proof of the Proposition 9.5 in \cite{BREZIS2011}, we conclude that $\chi_j(u) \in W^{1,p}_0(\Omega)$ and using \eqref{eq weak mean problem} with $\varphi =\chi_j(u)$, 
$$
\begin{aligned}
\int_{\Omega_j}|\nabla u|^pdx \leq & \int_{\Omega}\left(\ep^2 + |\nabla u|^{p-2}\right)\nabla u \cdot \nabla \chi_j(u)dx \\
=& \int_{\Omega}f(u)\chi_j(u)dx.
\end{aligned}
$$

Hypothesis (f1) yields the existence of $C>0$, depending only on $q$, such that  
$$
\int_{\Omega_j}|\nabla u|^pdx \leq 
 C\int_{\Omega_j} ( |u|^2 + |u|^{q} )dx
$$
and since $j\geq 1$ and $2<p<q<p^*$, we can select $C>0$ larger if necessary, but yet independent of $j$, so that 
$$
\begin{aligned}
\int_{\Omega_j}|\nabla u|^pdx \leq & C\int_{\Omega_j} |u|^{q}dx\\
\leq & 2^{q-1}C\left(\int_{\Omega_j} \left(|u| - j\right)^{q} dx + {\rm meas}(\Omega_j)j^q\right).
\end{aligned}
$$

Next, let $r\in (\frac{N}{p},\infty)$ such that $l:=\frac{rq}{r-1}\in (q,p^*)$. Using H\"{o}lder 's inequality and the fact $u\in L^l(\Omega)$, we estimate 
$$
\begin{aligned}
\int_{\Omega_j} \left(|u|-j\right)^q dx \leq & {\rm meas}(\Omega)^\frac{q}{r-1}\left(\int_{\Omega_j} \left(|u| - j\right)^l dx \right)^{\frac{q}{l}}\\
=& {\rm meas}(\Omega)^{\frac{q}{r-1}}\|u\|^{q-p}_{L^l(\Omega)} \left(\int_{\Omega_j}\left(|u| - j\right)^l dx \right)^{\frac{p}{l}}.   
\end{aligned}
$$

Therefore, 
\begin{equation}\label{THEO5.1URALTSEVA}
\int_{\Omega_j} |\nabla u|^pdx \leq C_1\left(\int_{\Omega_j}\left(|u| - j\right)^l dx \right)^{\frac{p}{l}} + {\rm meas}(\Omega_j)j^q,
\end{equation}
where $C_1:=2^{q-1} C
{\rm meas}(\Omega)^{\frac{q}{r-1}}\|u\|^{q-p}_{L^l(\Omega)}>0$. 

\medskip
Since \eqref{THEO5.1URALTSEVA} holds for any $j \in \mathbb{N}$, Theorem 5.1 from Chapter 2 and the footnote in page 71 both from \cite{LADYZHENSKAYAURALTSEVA} imply that $u\in L^{\infty}(\Omega)$. This concludes the proof of the lemma.   $ \blacksquare$
\end{proof}

\bigskip
For any $\ep\in \R$ arbitrary, consider the function $A_{\ep}:\R^N \to \R^N$ defined by 
$$
A_{\ep}(z):= \left[\ep^2+|z|^{p-2}\right]z \quad  \hbox{for every} \quad z\in \R^N.
$$

Observe that $A_{\ep}(0)=0$. Also $A_{0}\in C^0(\R^N)\cap C^{\infty}(\R^n-\{0\})$ and  $A_{\ep}\in C^{1}(\R^N)$.

\medskip
A direct computation yields that 
\begin{equation}\label{derivoperatorelliptic}
DA_{\ep}(z)=
\left\{
\begin{aligned}
 \left[\ep^2+|z|^{p-2}\right]I_{N\times N} + (p-2)|z|^{p-4}z\otimes z,& \quad \hbox{if} \quad z\neq 0\\
 \varepsilon^2 I_{N\times N},& \quad \hbox{if} \quad z=0,
\end{aligned}
\right.
\end{equation}
where $I_{N\times N}$ is the identity matrix and the symbol $\otimes$ represents the tensor product between two vectors of $\R^N$.

\medskip
\begin{lemma}\label{ellipticityAmu}
For any $\ep\in \R$ and any $z\in\R^N$,  the symmetric matrix $DA_{\ep}(z)$ has no negative eigenvalues and if $\varepsilon \neq 0$ they are strictly positive.
\end{lemma}

\begin{proof}
From \eqref{derivoperatorelliptic}, if $z=0$ the result is immediate with eigenvalues $\lambda_1=\cdots=\lambda_N=\ep^2$, so we assume that $z\neq 0$. Let $\lambda \in \R$ be such that
$$
{\rm det} \left[D A_{\ep}(z) - \lambda I_{N\times N}\right] =0.
$$

If $\lambda = \ep^2 + |z|^{p-2}$, then $\lambda \geq 0$ and it is strictly positive for $\varepsilon>0$. Assume next that $\lambda \neq \ep^2 + |z|^{p-2}$.

\medskip
From the matrix determinant formula, see for instance \cite{HarvilleD},  it follows that
$$
\begin{aligned}
{\rm det} \left[D A_{\ep}(z) - \lambda I_{N\times N}\right] &= {\rm det}\left[(\ep^2+|z|^{p-2} - \lambda)I_{N\times N} + (p-2)|z|^{p-4}z \otimes z \right] \\
&= {\rm det}\left[(\ep^2+|z|^{p-2} - \lambda ) I_{N \times N}\right]\left(1 +\frac{(p-2)|z|^{p-4}}{\ep^2+|z|^{p-2} - \lambda }z^t I_{N\times N}z\right)\\
&= \left[\ep^2+|z|^{p-2} - \lambda   \right]^{N-1}\left(\ep^2 + \left(p-1\right)|z|^{p-2} -\la\right). 
\end{aligned}
$$ 

We conclude that $\la = \ep^2 + \left(p-1\right)|z|^{p-2}$ so that $\lambda \geq 0$ and it is strictly positive if $\varepsilon>0$ proving the lemma. $ \blacksquare$
\end{proof}

\medskip

\begin{lemma}\label{hypothesesTolkdLieber}
Let $p > 2$. There exist constants $\gamma,\Gamma>0$ such that for any $\ep$ belonging to a bounded interval of $\R$, any $z\in \R^N$ and any $\zeta \in \R^N$, 
\begin{equation}\label{condtTolksLieber}
\gamma \left[\kappa + |z|\right]^{p-2}|\zeta|^2 \leq D A_{\ep}(z)\zeta \cdot \zeta \leq \Gamma \left[\kappa + |z|\right]^{p-2}|\zeta|^2,
\end{equation}
where $\kappa= \ep^{\frac{2}{p-2}}$.
\end{lemma}

\begin{proof}
Let $\ep$ belong to a fixed bounded interval of $\R$. Lemma \ref{ellipticityAmu} yields that $DA_{\ep}(z)$ is a semipositive definite matrix. Let $\la_1(z) \leq  \cdots \leq \la_N(z)$ denote the eigenvalues of $DA_{\ep}(z)$. 

\medskip
Lemma \ref{ellipticityAmu} implies that $\la_1, \la_N:\R^{N} \to \R$ are continuous in $z$. On the other hand, the characterization 

$$
\la_1(z)=\min \limits_{|\zeta| =1} DA_{\ep}(z)\zeta \cdot\zeta  \quad \hbox{and} \quad \la_N(z)=\max \limits_{|\zeta| =1} DA_{\ep}(z)\zeta\cdot\zeta,
$$
and Lemma \ref{ellipticityAmu} implies that for any $z\in \R^N$ and $\zeta \in \R^N$,
$$
(\ep^2 + |z|^{p-2})|\zeta|^2\,\leq \la_1(z)|\zeta|^2 \leq DA_{\ep}(z)\zeta\cdot\zeta \leq \la_N(z)|\zeta|^2 \,\leq (\ep^2 + (p-1)|z|^{p-2})|\zeta|^2.
$$

\medskip
Recall that for any $\beta>0$ and for any $a,b\geq 0$
\begin{equation}\label{reverseYoungIneq}
a^{\beta} + b^{\beta} \leq 2(a + b)^{\beta}.
\end{equation}
%where 
%$$
%C_{\beta}:=
%\left\{
%\begin{aligned}
%2^{\beta},& \quad \hbox{if} \quad \beta \in [0,1)\\
%1,& \quad \hbox{if} \quad \beta \in [1,\infty).
%\end{aligned}
%\right.
%$$

Using \eqref{reverseYoungIneq} and that $p > 2$, 
\begin{equation*}%\label{inequalityarithmetic1}
\begin{aligned}
(\ep^2 + (p-1)|z|^{p-2}) 
&\leq (p-1)\left[\ep^2 + |z|^{p-2}\right]\\
&\leq 2\left[\red{|\ep|}^{\frac{2}{p-2}}+|z|\right]^{p-2}
\end{aligned}
\end{equation*}
for some constant $C_p>0$ depending only on $p>2$. Consequently,
$$
D_{\ep}A(z)\zeta\cdot \zeta \leq 2\left[\red{|\ep|}^{\frac{2}{p-2}} +|z|\right]^{p-2}|\zeta|^2.
$$

\medskip
On the other hand, defining
$$
c_{\beta}:=
\left\{
\begin{aligned}
1,& \quad \hbox{if} \quad \beta \in [0,1)\\
2^{1-\beta},& \quad \hbox{if} \quad \beta \in [1,\infty)
\end{aligned}
\right.
$$
we get for $a,b\geq0$ that
$$c_\beta(a+b)^\beta\leq a^\beta+b^\beta.$$

Therefore, taking $\beta=p-2$ we get
$$
\begin{aligned}
\ep^2 + |z|^{p-2} \geq c_{p-2} \left[\red{|\ep|}^{\frac{2}{p-2}}+|z|\right]^{p-2}
\end{aligned}
$$
so that 
$$
D_{\ep}A(z)\zeta\cdot \zeta \geq c_{p-2}\left[\red{|\ep|}^{\frac{2}{p-2}} +|z|\right]^{p-2}|\zeta|^2.
$$

Set
$$
\gamma:= c_{p-2} \quad \hbox{and} \quad \Gamma:=2
$$
to obtain the inequalities in \eqref{condtTolksLieber}. This completes the proof.   $\blacksquare$
\end{proof}

\medskip
{\bf Remark:} From Lemma \ref{lemmaregul1}, any solution $u\in W^{1,p}_0(\Omega)$ of \eqref{eq mean problem} is bounded. By choosing  $\kappa, \gamma>0$ as in Lemma \ref{hypothesesTolkdLieber} and choosing
$$
\tilde{\Gamma}:= \Gamma + \sup \limits_{x\in \Omega}|f(u(x))|,
$$
conditions (1.4)-(1.7) in \cite{Tolksdorf84} are satisfied. These hypotheses correspond also to the condition (0.3a)-(0.3d) in \cite{Lieberman88}. 

\medskip

%Notice also that, 
%$$
%\frac{\tilde{\Gamma}}{\gamma}\leq C(p,q,\|u\|_{L^{\infty}(\Omega)})
%$$
%and so the constants in Theorem 1 in \cite{Tolksdorf84} and Theorem 1 in \cite{Lieberman88} are uniform for the operator $\ep^2 \Delta u + \Delta_p u = {\rm div}A_{\ep}(\nabla u)$ for $\ep\in [0,1]$.

% \medskip
%\textcolor{red}{IMPORTANT: A FAMILIY OF SOLUTIONS $u_{\ep}$ WITH UNIFORMLY BOUNDED NORMS $\|u_{\ep}\|_{W^{1,p}_0(\Omega)}= \mathcal{O}(1)$, THEN THE CONSTANT $\frac{\tilde{\Gamma}}{\gamma}>0$ CAN BE CHOSEN INDEPENDENT OF $\ep$, AND INDEPENDENT OF THE NORMS $\|u_{\\ep}\|_{L^{\infty}(\Omega)}$. THIS FOLLOWS ONCE AGAIN FROM TRACKING THE CONSTANTS IN THE PROOFS OF THE LEMMAS \ref{lemmaregul1} AND \ref{hypothesesTolkdLieber}.} 
%
%

\begin{lemma}\label{lemmaregul2}Assume the hypotheses in Lemma \ref{lemmaregul1} and let $\ep$ belong to any bounded interval of $\R$. Then, there exists $\beta=\beta(N,p,\Omega)\in (0,1)$ independent of $\ep>0$ such that any solution of $u\in W^{1,p}_0(\Omega)$ of \eqref{eq mean problem} belongs to $C^{1,\beta}(\overline{\Omega})\cap W^{2,2}_{loc}(\Omega)$.
\end{lemma}

\begin{proof}
From Lemma \ref{lemmaregul1}, $u\in L^{\infty}(\Omega)$. Let $\ep\in \R$ be fixed and set $h:=f(u(\cdot))$. The previous remark states that the operator ${\rm div}A_{\ep}(\nabla u)$ and the right-hand side $h$ satisfy the hypotheses in the [Theorem 1,\cite{Tolksdorf84}] and in [Theorem 1, \cite{Lieberman88}].

\medskip
Since, any bounded solution $u$ of \eqref{eq mean problem} solves weakly the equation 
$$
- \ep^2 \Delta u - \Delta_p u = h \quad \hbox{in} \quad \Omega, \qquad u=0 \quad \hbox{on} \quad \partial \Omega,
$$
there exists $\beta=\beta(N,p,\Omega)>0$ such that $u\in C^{1,\beta}(\overline{\Omega})$ and there exists
 $$C=C(N,p,\|u\|_{L^{\infty}(\Omega)},\Omega)>0$$
  such that
\begin{equation*}%\label{crucialestimateII}
\|u\|_{C^{1,\beta}(\overline{\Omega})} \leq C.
\end{equation*}

From Proposition 1 in \cite{Tolksdorf84} it follows that $u\in W^{2,2}_{loc}(\Omega)$ and this completes the proof.
 \medskip
%\textcolor{red}{IMPORTANT: HERE AGAIN, A FAMILIY OF SOLUTIONS $u_{\ep}$ WITH UNIFORMLY BOUNDED NORMS $\|u_{\ep}\|_{W^{1,p}_0(\Omega)}= \mathcal{O}(1)$, THEN THE CONSTANT $C>0$ IN \eqref{crucialestimateII} CAN BE CHOSEN INDEPENDENT OF $\ep$, AND INDEPENDENT OF THE NORMS $\|u_{\ep}\|_{L^{\infty}(\Omega)}$. THIS FOLLOWS ONCE AGAIN FROM TRACKING THE CONSTANTS IN THE PROOFS OF THE LEMMAS \ref{lemmaregul1} AND \ref{hypothesesTolkdLieber} AND FROM THE CONCLUSION IN Theorem 1, \cite{Lieberman88}. THIS IS CRUCIAL TO PROVE THE CONVERGENCE RESULT IN THE NEXT SECTION.} 
$\blacksquare$
\end{proof}

%\medskip
%\begin{lemma}\label{technicallemma}
%For any $\beta>0$ the functions 
%$$
%\R^N \ni z\mapsto |z|^{\beta} \in \R,\qquad \hbox{and} \qquad \R^N \ni z\mapsto |z|^{\beta}z_iz_j 
%$$
%belong to $C^{0,\beta}_{loc}(\R^N)$.
%\end{lemma}

%\begin{proof}
%Take arbitrary $z,w\in \R^N$ with $z\neq w$. Consider the smooth function 
%$$
%g(\zeta):= |z + \zeta (w-z)|^2 \quad \hbox{for}\quad \zeta \in \R.
%$$

%Since,
%$$
%g(\zeta)= (1-\zeta)^2 |z|^2 + 2(1-\zeta)\zeta z\cdot w + \zeta^2 |w|^2,
%$$
%it follows that
%$$
%\frac{d g}{d \zeta}(\zeta)= -2(1-\zeta)|z|^2 + 2(1-2\zeta)z\cdot w + 2 \zeta |w|^2 
%$$
%and
%$$
%\frac{d^2 g}{d \zeta^2}(\zeta)= 2|z|^2 -4z\cdot w +2  |w|^2 = 2|z-w|^2 > 0.
%$$

%So that $g$ is strictly convex and therefore there exists a unique $\zeta_0 >0$ such that
%$$
%g(\zeta_0):= \min \limits_{\zeta \in \R} g(\zeta).
%$$  

%Let $\theta=\theta(z,z-w)\in (0,\pi]$ be the absolute angle between $z$ and $z-w$. Since,
%$$
%\zeta_0 = \frac{z\cdot (z-w)}{|z-w|^2},
%$$
%we find that
%$$
%\begin{aligned}
%g(\zeta_0)&= \left|z + \frac{z \cdot (z-w)}{|z-w|^2} (w-z)\right|^2
%\\
%&= |z|^2 (1 - \cos^2 \theta)
%\end{aligned}
%$$

%We conclude that for any $z,w\in \R^N$ with %$z \neq w$ and for any $\zeta \in \R$
%$$
%|z + \zeta(w-z)| \geq |z|.   $ \blacksquare$
%$$
%\end{proof}  

\bigskip
\begin{proposition}\label{regularityoftheNehariequator}
Let $p> 2$ and let $\ep \neq 0$ belong to a bounded interval of $\R$. Assume hypotheses from Lemma \ref{lemmaregul1} and let $u\in W^{1,p}_0(\Omega)$ a solution of \eqref{eq mean problem}. Then, for any $r>1$, $u\in W^{2,r}(\Omega)$.
\end{proposition}

\begin{proof}
From Lemma \ref{lemmaregul2}, $u\in C^{1,\beta}(\overline{\Omega})\cap W^{2,2}_{loc}(\Omega)$. Set $h:=f(u(\cdot))$.

\medskip
Using hypothesis (f1), for $x,y\in \Omega$,
$$
|h(x)-h(y)|=|f(u(x))-f(u(y))|\leq C|u(x)-u(y)|\leq \tilde{C}|x-y|^{\beta}
$$
so that $h\in C^{0,{\beta}}(\overline{\Omega})$.

\medskip
Consider the BVP for $w$,
\begin{equation}\label{quasilinearapprxmu}
\left\{
\begin{aligned}
-\left(\ep^2 + |\nabla u|^{p-2}\right)\Delta w- (p-2)|\nabla u|^{p-4}\partial_i u\partial_j u \partial_{ij} w &= h \quad \hbox{in}\quad \Omega\\
w&=0 \quad \hbox{on}\quad \partial \Omega,
\end{aligned}
\right.
\end{equation}
where summation over repeated indices is understood.

\medskip
Observe that $u\in W^{2,2}_{loc}(\Omega)\cap C^{1,\beta}(\overline{\Omega})$ is a strong solution of \eqref{quasilinearapprxmu}. Set 
$$
a_{ij,\ep}:= \left(\ep^2 + |\nabla u|^{p-2}\right)\delta_{ij} + (p-2)|\nabla u|^{p-4}\partial_iu\partial_j u
$$
so that equation \eqref{quasilinearapprxmu} reads as 
\begin{equation*}%\label{quasilinearapprxmu1}
-a_{ij,\ep}(x)\partial_{ij}w =h \quad \hbox{in} \quad \Omega,  \qquad w=0 \quad \hbox{on} \quad \partial \Omega.
\end{equation*}

\medskip
Since $p> 2$ and $\ep^2 > 0$, $a_{ij,\ep} \in C(\overline{\Omega})$.  Also, Lemma \ref{hypothesesTolkdLieber} implies that $a_{ij,\ep}$ is strictly elliptic. The fact that $h\in L^{\infty}(\Omega)$ and Theorem 9.15 in \cite{GilbargTrudinger} yield that given any $r\geq 2$, there exists a unique strong solution $w\in W^{2,r}(\Omega)$ of \eqref{quasilinearapprxmu}.

\medskip
From the Theorem 8.9 in \cite{GilbargTrudinger} and the remark after it, we conclude that $w$ is the unique strong solution of \eqref{quasilinearapprxmu} in $W^{2,2}_{loc}(\Omega)\cap W^{1,2}_0(\Omega)$. Therefore, $w=u$ a.e. in $\Omega$ and so it is also unique in $W^{2,r}(\Omega)$. Since $\Omega$ is a bounded domain, $u\in W^{2,r}(\Omega)$ for any $r>1$ and this proves the result.  $\blacksquare$
\end{proof}  

The following lemma provides a useful result for the computations presented in the next section (see the proof of Lemma \ref{dif} below).

\medskip
%\blue{In this part, I have removed the loc-subindex from the Sobolev spaces. }

\begin{lemma}\label{lemma p-laplacian}
If $p\geq 2$, then for any $\phi \in W^{1,p}(\Omega)\cap W^{2,2}(\Omega)$,
\begin{equation*}%\label{eq p-laplacian}
	\Delta_p \phi =\vert \nabla \phi\vert^{p-2}\Delta \phi+(p-2)\vert \nabla \phi \vert ^{p-4}\sum_{i,j=1}^{N}\partial_i \phi\partial_j \phi \partial_{ij}\phi
\end{equation*}
a.e. in $\Omega$. In particular, if $\phi \in W^{1,p}(\Omega)\cap W^{2,p}(\Omega)$ then  $\Delta_p \phi \in L^{p'}(\Omega)$, where $p'=\frac{p}{p-1}$. 
\end{lemma}

\begin{proof}
To verify this claim, we first let $\mu>0$ be arbitrary, but fixed. Next, integrating by parts we find that for any $\psi \in C^{\infty}_{c}(\Omega)$,
\begin{equation}\label{distrpLap}
\int_{\Omega} \left[\mu^2 + |\nabla \phi|^2\right]^{\frac{p-2}{2}}\nabla \phi \cdot \nabla \psi dx = - \int_{\Omega} {\rm div}\left(\left[\mu^2 + |\nabla \phi|^2\right]^{\frac{p-2}{2}}\nabla \phi \right) \psi dx.
\end{equation}
 
We estimate pointwise the integrand in the left-hand side of \eqref{distrpLap},  to find that
$$
\begin{aligned}
\left[\mu^2 + |\nabla \phi|^2\right]^{\frac{p-2}{2}}|\nabla \phi| | \nabla \psi| &\leq C\left[\mu^{p-2}+ |\nabla \phi|^{p-2}\right]|\nabla \phi||\nabla \psi|\\
&= C \left[\mu^{p-2}|\nabla \phi| + |\nabla \phi|^{p-1}\right]\|\nabla \psi\|_{L^{\infty}(\Omega)}
\end{aligned}
$$
a.e. in $\Omega$.

\medskip
On the other hand, since $\phi \in W^{2,2}(\Omega)$, we compute
\begin{multline*}
 {\rm div}\left(\left[\mu^2 + |\nabla \phi|^2\right]^{\frac{p-2}{2}}\nabla \phi \right) = \left[\mu^2 + |\nabla \phi|^2\right]^{\frac{p-2}{2}}\Delta \phi \\
 + (p-2)\left[\mu^2 + |\nabla \phi|^2\right]^{\frac{p-4}{2}}\sum_{i,j=1}^N\partial_i \phi \partial_j\phi \partial_{ij} \phi 
\end{multline*}
a.e in $\Omega$. Using this computation and Young inequality, we estimate
$$
\begin{aligned}
\left|{\rm div}\left(\left[\mu^2 + |\nabla \phi|^2\right]^{\frac{p-2}{2}}\nabla \phi \right)\right|& \leq C[\mu^{p-2} + |\nabla \phi|^{p-2}]|D^2 \phi|\\
&\leq C\left[\mu^{p-2}|D^2 \phi| + \frac{p-2}{p}|\nabla \phi|^p + \frac{2}{p}|D^2 \phi|^{\frac{p}{2}}\right]
\end{aligned}
$$
a.e. in $\Omega$. Therefore, a direct application of the {\it Dominated Convervenge Theorem} taking $\mu \to 0^+$ yields that 
$$
\begin{small}
\int_{\Omega} |\nabla \phi|^{p-2}\nabla \phi \cdot \nabla \psi dx = \int_{\Omega}\left(|\nabla\phi|^{p-2}\Delta \phi + (p-2)|\nabla \phi|^{p-4}\sum_{i,j=1}^N \partial_{i} \phi\partial_{j}\phi \partial_{ij}\phi\right) \psi dx. 
\end{small}
$$

Since $\psi \in C^{\infty}_{c}(\Omega)$ is arbitrary, we conclude that
$$
\Delta_p \phi = |\nabla\phi|^{p-2}\Delta \phi + (p-2)|\nabla \phi|^{p-4}\sum_{i,j=1}^N \partial_{i} \phi\partial_{j}\phi \partial_{ij}\phi
$$
a.e. in $\Omega$ and this proves the first part of the claim. Now, if $\phi \in W^{1,p}(\Omega)\cap W^{2,p}(\Omega)$ then  a direct application of H\"{o}der's inequality shows that $\Delta_p \phi \in L^{p'}(\Omega)$.   $ \blacksquare$
\end{proof}

\section{Morse index in the case $\ep \neq 0$}\label{computMorseIndexNondeg} 

%\blue{Nos falta establecer claramente las hip\'otesis que supondremos en esta sección.}

Let $\ep \neq 0$ and let $u_{\ep} \in \MM_{\ep}$ be a local minimizer for $J_{\varepsilon}|_{\MM_\varepsilon}$. Since $\mathcal{M}_{\ep}$ is not a smooth manifold of $W_0^{1,p}(\Omega)$, we cannot infer directly the local behavior of $J_{\varepsilon}|_{\MM_\varepsilon}$ around $u_{\ep}$ to estimate its Morse index. 
We overcome this issue by adapting the approach in \cite{BartschWeth2003} to our setting and by working on a suitable dense subspace of $W^{1,p}_0 (\Omega)$ containing $u_{\ep}$.

\medskip

Consider the functionals $I_{\ep},H:W_0^{1,p}(\Omega)\to \R$ defined by 
\begin{equation*}%\label{deffuncPQ}
I_{\ep}(v):= \int_{\Omega}\left(\frac{\ep^2}{2}|\nabla v|^2 + \frac{1}{p}|\nabla v|^p \right)dx , \qquad  H(v):=\int_{\Omega}F(v)dx
\end{equation*}
for every $v\in W^{1,p}_0(\Omega)$.

\medskip
Observe that $J_{\ep}=I_{\ep}-H$ and $I_{\ep},H\in C^2(W_0^{1,p}(\Omega))$
with
$$
DI_{\ep}(v)\varphi = \int_{\Omega}\left(\ep^2 + |\nabla v|^{p-2}\right)\nabla v \cdot \nabla \varphi dx \quad \hbox{and} \quad DH(v)\varphi =\int_{\Omega} f(v)\varphi dx
$$  
for every $v,\varphi \in W^{1,p}_0(\Omega)$.

\medskip
Consider also the Banach space $W:=W_0^{1,p}(\Omega)\cap W^{2,p}(\Omega)$ endowed with the norm of $W^{2,p}(\Omega)$.

\medskip

Our first lemma is a technical result concerned with the regularity of $I_{\ep}$ and $H$ in $W$. This lemma will be used to prove that $\MM_{\ep} \cap W$ is a $C^1$-manifold embedded in $W$.

\begin{lemma}\label{dif}
Let $f$ satisfy {\it (f1)} and let $\ep \neq 0$. Consider the functions $P^\pm_{\ep}, Q^\pm:W_0^{1,p}(\Omega)\to \R$ defined by:
$$
P^\pm_{\ep}(v):=DI_{\ep}(v) (v^\pm) ,\qquad Q^\pm(v):=DH(v)(v^\pm)
$$
for every $v\in W^{1,p}_0(\Omega)$. Then,
\begin{enumerate}[label=(\alph*)]	
\item $P_{\ep}^\pm|_W\in C^1(W)$. Moreover, for every $v\in W$, $DP^\pm_{\ep}(v) \in (W_0^{1,p}(\Omega))^*$ with 
\begin{equation}\label{DPep}
\begin{aligned}
DP^\pm_{\ep}(v)\varphi &=\int_{\{\pm v>0\}}(- \ep^2 \Delta v -\Delta_p v)\varphi dx\\
& + \int_{\{\pm v>0\}} ( \ep^2 \nabla v + (p-1)|\nabla v|^{p-2}\nabla v) \cdot \nabla \varphi dx, 
\end{aligned}
\end{equation}
for every $\varphi \in W_0^{1,p}(\Omega)$.

\medskip		
\item  $Q^\pm\in C^1(W_0^{1,p}(\Omega))$ and
%$Q^\pm|_W\in C^1(W)$ with 
		\begin{equation*}
DQ^\pm(v)\varphi=\int_{\Omega}f'(v^\pm)v^\pm \varphi+\int_{\Omega} f(v^\pm)\varphi \qquad \forall  v,\varphi \in W_0^{1,p}(\Omega). 
\end{equation*}
\end{enumerate}
\end{lemma}

\begin{proof} We proceed as in the proof of the Lemma 3.1 in \cite{BartschWeth2003}.
First, we prove {\it (a)}. H\"{o}lder's inequality implies that for each $v\in W$, $P_{\ep}^{\pm}(v)$ given by \eqref{DPep} is well defined and continuous in $W_0^{1,p}(\Omega)$. 

\medskip
Let $v,\varphi \in W$ and $t\in \R-\{0\}$ be arbitrary. Consider the sets
$$
\begin{aligned}
C_1^t:= \{v + t \varphi \geq 0, \, \, v>0\},& \qquad C_2^t:= \{v + t \varphi \geq 0, \, \, v<0\},\\
C_3^t:= \{v + t \varphi <0, \, \, v>0\},& \qquad C_4:= \{v =0\},
\end{aligned}
$$
with corresponding characteristic functions $\mathds{1}_{C^t_1}, \mathds{1}_{C^t_2},\mathds{1}_{C^t_3},\mathds{1}_{C_4}: \Omega \to \R$.

\medskip
Observe that
$$
P^+_{\ep}(v+t\varphi)-P^+_{\ep}(v)
$$
$$
\begin{aligned}
= & \int_{\Omega}\left([\ep^2+|\nabla (v+t\varphi)|^{p-2}]\nabla (v+t\varphi)\cdot \nabla (v+t\varphi)^+ -[\ep^2+|\nabla v|^{p-2}]\nabla v\cdot \nabla v^+ \right) dx\\
= & \int_{\Omega} \ep^2 \left(\nabla (v+ t\varphi) \cdot \nabla (v + t\varphi)^+ -\nabla v\cdot \nabla v^+ \right) dx + \\
& + \int_{\Omega} \left(|\nabla (v+ t\varphi)|^{p-2} \nabla (v+ t\varphi) \cdot \nabla (v+ t\varphi)^+ - |\nabla v|^{p-2}\nabla v \cdot \nabla v^+\right)dx.
\end{aligned}
$$

After lenghty, but straightforward computations involving integration by parts, the use of Lemma \ref{lemma p-laplacian} and rearranging of the terms, we find that
$$
\begin{aligned}
P^+_{\ep}(v+t\varphi)-P^+_{\ep}(v)= & \int_{\Omega} \left(- \ep^2 \Delta v - \Delta_p v \right)[ (v + t\varphi)^+ - v^+ ] dx\\
&+ \int_{\Omega} \left( |\nabla (v + t \varphi)|^{p-2} - |\nabla v|^{p-2}\right)\nabla v \cdot \nabla ( v+ t\varphi)^+ dx  \\
&+ t \int_{\Omega} \left(\ep^2 \nabla (v + t \varphi)^+  + |\nabla v|^{p-2} \nabla (v+ t \varphi)^+\right)\cdot \nabla \varphi dx \\
&+ t\int_{\Omega} \left( |\nabla (v + t \varphi)|^{p-2} - |\nabla v|^{p-2}\right)\nabla \varphi \cdot \nabla ( v+ t\varphi)^+ dx\\
=:& I_{t} + II_{t} + III_{t} + IV_{t}.
\end{aligned}
$$

From Lemma 7.7 of \cite{GilbargTrudinger} and Lemma \ref{lemma p-laplacian} in Section 3, $\nabla v$, $\Delta v$ and $\Delta_p v$ are zero a.e. on $C_4$. 

First we estimate $I_t$. Observe that
$$
\begin{aligned}
I_{t}:= & t \int_{C_1^t}(-\ep^2 \Delta v - \Delta_p v) \varphi dx + \int_{C_2^t} (-\ep^2 \Delta v - \Delta_p v) (v + t \varphi) dx+\int_{C_3^t} (\ep^2 \Delta v + \Delta_p v) v  dx\\
=& t \int_{\Omega}\left(\mathds{1}_{C_1^t}+\mathds{1}_{C_2^t}\right)(-\ep^2 \Delta v - \Delta_p v) \varphi dx+\int_{\Omega}\left(\mathds{1}_{C_3^t}-\mathds{1}_{C_2^t}\right)(\ep^2 \Delta v + \Delta_p v) v dx.
\end{aligned}
$$

Notice that either in $C_2^t$ or in $C_3^t$ we have that $|v|\leq |t||\varphi|$ implying that 
$$\left|\left(\mathds{1}_{C_3^t}-\mathds{1}_{C_2^t}\right)(\ep^2 \Delta v + \Delta_p v) v\right|\leq |t|\left(\mathds{1}_{C_3^t}+\mathds{1}_{C_2^t}\right)(\ep^2 |\Delta v|+ |\Delta_p v)|) |\varphi | $$
Also, since for $i=1,2,3$
$$
|\mathds{1}_{C_i^t}(-\ep^2 \Delta v - \Delta_p v) \varphi | \leq (\ep^2 |\Delta v|+ |\Delta_p v)|) |\varphi | 
$$
and as $t \to 0$, $\mathds{1}_{C_1^t} \to \mathds{1}_{\{v>0\}}$ and for $i=2,3$ $\mathds{1}_{C_i^t} \to 0$ a.e. in $\Omega$, the {\it Dominated Convergence Theorem} yields
\begin{equation}\label{I}
\lim \limits_{t\to 0} \frac{1}{t}I_{t} = \int_{\{v>0\}} (- \ep^2 \Delta v - \Delta_p v)\varphi dx.  
\end{equation}

Next, we estimate $II_{t}$. Since $p>2$,
$$
\begin{aligned}
II_{t} =&t \int_{\Omega} \left((p-2)\int_0^1 |\nabla v + \zeta t \varphi|^{p-4}\nabla (v + \zeta t \varphi)\cdot \nabla \varphi d\zeta \right)\nabla v\cdot \nabla (v+ t\varphi)^+dx.
\end{aligned}
$$

Thus,
$$
\begin{aligned}
\frac{1}{t}II_{t} =& \int_{\Omega} \left((p-2)\int_0^1 |\nabla\left( v + \zeta t \varphi\right)|^{p-4}\nabla (v + \zeta t \varphi)\cdot \nabla \varphi d\zeta \right)\nabla v\cdot \nabla (v+ t\varphi)^+dx \\ 
=& \int_{\Omega} \mathds{1}_{C_1^t}\left((p-2)\int_0^1 |\nabla \left( v + \zeta t \varphi\right)|^{p-4}\nabla (v + \zeta t \varphi)\cdot \nabla \varphi d\zeta \right)\nabla v\cdot \nabla (v+ t\varphi)dx\\
&+ \int_{\Omega} \mathds{1}_{C_2^t}\left((p-2)\int_0^1 |\nabla \left( v + \zeta t \varphi\right)|^{p-4}\nabla (v + \zeta t \varphi)\cdot \nabla \varphi d\zeta \right)\nabla v\cdot \nabla (v+ t\varphi)dx.
\end{aligned}
$$

Notice that there exists $C_p>0$ such that for $i=1,2$
$$
\left| \mathds{1}_{C_i^t}\left((p-2)\int_0^1 |\nabla \left( v + \zeta t \varphi\right)|^{p-4}\nabla (v + \zeta t \varphi)\cdot \nabla \varphi d\zeta \right)\nabla v\cdot \nabla (v+ t\varphi)\right| 
$$
$$
\leq  C_p(|\nabla v|^{p-2} + |t|^{p-2}|\nabla \varphi|^{p-2}) |\nabla v||\nabla \varphi|.
$$

Also, since
$$
 \mathds{1}_{C_1^t}\left((p-2)\int_0^1 |\nabla \left( v + \zeta t \varphi\right)|^{p-4}\nabla (v + \zeta t \varphi)\cdot \nabla \varphi d\zeta \right)\nabla v\cdot \nabla (v+ t\varphi) \to \mathds{1}_{\{v>0\}}|\nabla v|^{p-2}\nabla v \cdot \nabla \varphi
$$
and
$$
 \mathds{1}_{C_2^t}\left((p-2)\int_0^1 |\nabla \left( v + \zeta t \varphi\right)|^{p-4}\nabla (v + \zeta t \varphi)\cdot \nabla \varphi d\zeta \right)\nabla v\cdot \nabla (v+ t\varphi)^+ \to 0
$$
a.e. in $\Omega$, as $t\to 0$, we conclude that 
\begin{equation}\label{II}
\lim \limits_{t\to 0} \frac{1}{t}II_{t} = \int_{\{v>0\}}(p-2)|\nabla v|^{p-2}\nabla v\cdot \nabla \varphi dx.
\end{equation}

As for $III_{t}$, observe that
$$
\begin{aligned}
\frac{1}{t}III_{t}=& \int_{\Omega} \mathds{1}_{C_1^t}\left(\ep^2 \nabla (v + t \varphi)  + |\nabla v|^{p-2} \nabla (v+ t \varphi)\right)\cdot \nabla \varphi dx \\
&+ \int_{\Omega} \mathds{1}_{C_2^t}\left(\ep^2 \nabla (v + t \varphi)  + |\nabla v|^{p-2} \nabla (v+ t \varphi)\right)\cdot \nabla \varphi dx\\
&+ \ep^2\int_{\Omega} \mathds{1}_{C_4^t} \nabla (v + t \varphi)^+ \cdot \nabla \varphi dx  
\end{aligned}
$$
and proceeding in a similar fashion as above, 
\begin{equation}\label{III}
\lim \limits_{t \to 0} \frac{1}{t}III_{t} = \int_{\{v>0\}} \left(\ep^2 \nabla v  + |\nabla v|^{p-2}\nabla v\right)\cdot \nabla \varphi dx.
\end{equation}

Finally, a direct application of the {\it Dominated Convergence Theorem} proves that
\begin{equation}\label{IV}
\lim \limits_{t \to 0} \frac{1}{t}IV_{t} = 0.
\end{equation}

From \eqref{I},\eqref{II}, \eqref{III} and \eqref{IV}, and using that $v, \varphi$ are arbitrary, we conclude that $DP_{\ep}^+$ is Gateaux differentiable in $W$ and \eqref{DPep} is satisfied in the Gateaux sense.

\medskip
Next, we prove that $P^+_{\ep}|_W\in C^1(W)$. Let $v_n,v\in W$ be such that $v_n\to v$ strongly in $W$ and let $\varphi \in W$ be such that $\|\varphi\|_{W}=1$. With no loss of generality assume that $v_n(x) \to v(x)$ and $\nabla v_n(x) \to \nabla v(x)$ as $n\to \infty$ for a.e. $x\in \Omega$. Then, 
$$
\begin{aligned}
|(D P^+_{\ep}(v_n)-DP^+_{\ep}(v))\varphi| \leq& \ep^2\int_{\Omega}\left(|\Delta v_n-\Delta v||\varphi| + |\nabla (v_n -v)||\nabla \varphi|\right) dx\\
&+\int_{\Omega} \left(|\Delta_p v_n-\Delta_p v|\right)|\varphi| dx \\
&+(p-1)\int_{\Omega} \left||\nabla v_n|^{p-2}\nabla v_{n}-|\nabla v|^{p-2}\nabla v|\right| |\nabla \varphi|dx \\
&+ \int_{\{v\leq 0 < v_n\}} \ep^2\left(|\Delta v_n||\varphi| + |\nabla v_n||\nabla \varphi|\right)dx \\
&+ \int_{\{v_n\leq 0 < v\}} \ep^2\left(|\Delta v||\varphi| + |\nabla v||\nabla \varphi|\right)dx\\
&+ \int_{\{v\leq 0 < v_n\}} \left(|\Delta_p v_n||\varphi| + (p-1)|\nabla v_n|^{p-1}|\nabla \varphi|\right)dx \\
&+ \int_{\{v_n\leq 0 < v\}} \left(|\Delta_p v||\varphi| + (p-1)|\nabla v|^{p-1}|\nabla \varphi|\right)dx.  
\end{aligned}
$$

Recall that $v_n, v \in W$. Using H\"{o}der's inequality, we find a constant $C_p>0$ independent of $n$ such that
\begin{equation}\label{A1}
\ep^2\int_{\Omega}\left(|\Delta v_n-\Delta v||\varphi| + |\nabla (v_n -v)||\nabla \varphi|\right) dx \leq \ep^2 C_p\|v_n -v\|_{W} \|\varphi\|_{W}.
\end{equation}

Proceeding in the same fashion as above, using Lemma \ref{lemma p-laplacian} and taking $C_p>0$ larger if necessary,
\begin{equation}\label{A2}
\int_{\Omega} \left(|\Delta_p v_n-\Delta_p v|\right)|\varphi| dx \leq C_p \|\varphi\|_{L^{p}(\Omega)} \|\Delta_p v_n-\Delta_p v\|_{L^{\frac{p}{p-1}}(\Omega)} \longrightarrow 0,
\end{equation}
 as $n \to \infty$.\\

Arguing in the same fashion as in \eqref{A2}, we conclude also that as $n \to \infty$,
\begin{equation*}%\label{A3}
\int_{\Omega} \left||\nabla v_n|^{p-2}\nabla v_{n}-|\nabla v|^{p-2}\nabla v|\right| |\nabla \varphi|dx \leq o(1)\|\varphi\|_{W}.
\end{equation*}

We estimate the last four integrals as follows. Cauchy-Schwarz, H\"{o}lder's inequality and the fact that $p>2$ imply that 
\begin{multline*}
\int_{\{v\leq 0 < v_n\}} \ep^2\left(|\Delta v_n||\varphi| + |\nabla v_n||\nabla \varphi|\right)dx \\
\leq \left(\int_{\{v\leq 0 < v_n\}} \ep^2\left(|\Delta v_n|^2 + |\nabla v_n|^2\right)dx\right)^{\frac{1}{2}}C_p\|\varphi\|_{W}.
\end{multline*}

Since $v_n(x) \to v(x)$ for a.e. $x \in \Omega$ and $\nabla v, \Delta v =0$ a.e. in the set $\{v=0\}$, the Dominated Convergence Theorem yields
\begin{equation}\label{A4}
\int_{\{v\leq 0 < v_n\}} \ep^2\left(|\Delta v_n||\varphi| + |\nabla v_n||\nabla \varphi|\right)dx \leq o(1)\|\varphi\|_{W}.
\end{equation}

Similarly we conclude that 
\begin{equation}\label{A5}
\int_{\{v_n\leq 0 < v\}} \ep^2\left(|\Delta v||\varphi| + |\nabla v||\nabla \varphi|\right)dx \leq o(1)\|\varphi\|_{W}.
\end{equation}

On the other hand, using H\"{o}lder's inequality and the fact that $\Delta_p v \in L^{\frac{p}{p-1}}(\Omega)$ we get
\begin{multline*}
\int_{\{v\leq 0 < v_n\}} \left(|\Delta_p v_n||\varphi| + (p-1)|\nabla v_n|^{p-1}|\nabla \varphi|\right)dx \\
\leq  C_p \left(\int_{\{v\leq 0 < v_n\}} \left(|\Delta_p v_n|^{\frac{p}{p-1}}+ |\nabla v_n|^{p}\right)dx\right)^{\frac{p-1}{p}}\|\varphi\|_{W}
\end{multline*}
and arguing as above, as $n \to \infty$
\begin{equation}\label{A6}
\int_{\{v\leq 0 < v_n\}} \left(|\Delta_p v_n||\varphi| + (p-1)|\nabla v_n|^{p-1}|\nabla \varphi|\right)dx \leq o(1)\|\varphi\|_{W}.
\end{equation}

The same argument yields 
\begin{equation}\label{A7}
\int_{\{v_n\leq 0 < v\}} \left(|\Delta_p v||\varphi| + (p-1)|\nabla v|^{p-1}|\nabla \varphi|\right)dx \leq o(1)\|\varphi\|_{W}.
\end{equation} 

Putting together \eqref{A1}-\eqref{A7},
$$
|(D P^+_{\ep}(v_n)-DP^+_{\ep}(v))\varphi|  \leq o(1)\|\varphi\|_{W}
$$
as $n \to \infty$ and since $\varphi \in W$ with $\|\varphi\|_{W} =1$ is arbitrary, we conclude that $DP_{\ep}^+$ is continuous.

\medskip
The same argument with obvious changes yields also that $P_{\ep}^-$ belongs to $C^1(W)$ and \eqref{DPep} holds true. 

\medskip
The proof of {\it (b)} goes along the same lines as the proof of {\it (a)}. This concludes the proof.  $\blacksquare$
\end{proof}

%\medskip
%On the other hand, consider the sets $A^{\pm}:=\{v\in W\,:\, v^{\pm}= 0 \quad \hbox{a.e. in} \quad \Omega\}$. Observe that $A^{+},A^-$ are closed sets in $W$. To prove this claim, let $\{v_n\}_{n\in\N }$ in $A^{\pm}$ such that $v_n\to v$ strongly in $W$. Then $v_n\to v$ in $W_0^{1,p}(\Omega)$ and therefore $v_n^{\pm} \to v^{\pm}$ in $W_0^{1,p}$, implying that $v^{\pm}=0$ a. e. in $\Omega$.

\medskip

\textbf{Remark:} We point out that all the computations in the previous proof hold true taking $\varphi \in W_0^{1,p}(\Omega)$ with $\|\varphi\|_{W^{1,p}_0(\Omega)}=1$.

\begin{lemma}\label{lemma manifold}
Let $f$ satisfy {\it (f1)} and {\it(f4)} and let $\ep \neq 0$. Assume that $\mathcal{M}_{\ep}\cap W\neq \emptyset$. Then, $\MM_{\ep}\cap W$ is a $C^1-$manifold in $W$ of codimension two.
\end{lemma}

\begin{proof}
Following the notations from Lemma \ref{dif}, recall that 
$$\mathcal{M}_{\ep}\cap W=\{v\in W\,:\,v^{\pm}\neq 0 \quad \hbox{and} \quad P^{\pm}_{\ep}(v)- Q^{\pm}(v)=0\}.
$$ 

Since $\MM_{\ep}\cap W$ is contained in an open subset of $W$, in view of the Lemma \ref{dif} and the {\it Implicit Function Theorem}, it suffices to prove that given any $v\in \MM_{\ep}\cap W$, the function $(D\left(P^{+}_{\ep}- Q^{+}\right)(v),D\left(P^{-}_{\ep}- Q^{-}\right)(v))$, which belongs to $B(W,\R^2)$, is a surjective function.

\medskip
Fix $v\in \MM_{\ep}\cap W$ and let $\lambda,\eta \in \R$. Observe that
 $$(D\left(P^{+}_{\ep}- Q^{+}\right)(v)(\lambda v^++\eta v^-),D\left(P^{-}_{\ep}- Q^{-}\right)(v)(\lambda v^++\eta v^-))$$
  can be written in matrix form as
\begin{equation}\label{eq mat}
\begin{pmatrix}
D(P_{\ep}^+ - Q^+)(v) v^+  & 0         \\[0.2em]
	0 & D(P_{\ep}^- - Q^-)(v) v^-          \\[0.2em]
	\end{pmatrix}
	\begin{pmatrix}
	\lambda         \\[0.3em]
	\eta           \\[0.3em]
	\end{pmatrix}.
	\end{equation}
Using that $v\in W$ and integrating by parts, 
$$
\begin{aligned}
D(P_{\ep}^+ - Q^+)(v)v^+= & \int_{\Omega}(2 \ep^2 |\nabla v^+|^{2}+ p|\nabla v^+|^{p}-f'(v^+)(v^+)^2 - f(v^+)v^+)dx\\
D(P_{\ep}^- - Q^-)(v) v^-= & \int_{\Omega}(2 \ep^2 |\nabla v^-|^{2}+ p|\nabla v^-|^{p}-f'(v^-)(v^-)^2 - f(v^-)v^-)dx.
\end{aligned}
$$

Using that $v\in \mathcal{M}_{\ep}$ and (f4), 
\begin{equation*}
\begin{aligned}
D(P^{\pm}_{\ep} - Q^{\pm})(v)v^\pm = &-(p-2)\ep^2\int_{\Omega}|\nabla v^\pm|^{2}dx\\
& \hspace{1.5cm}+ \int_{\{v^\pm\neq 0\}}\left((p-1)\frac{f(v^\pm)}{v^\pm}-f'(v^\pm)\right)(v^\pm)^2 dx\\
<&0.
\end{aligned}
\end{equation*}

Therefore, the matrix in \eqref{eq mat} is invertible. From this, we can not directly conclude the desired surjectivity since the functions $v^+,v^-$ do not necessarily belong to $W$. We overcome this obstacle by considering the continuous function $D_v: W_0^{1,p}(\Omega) \times W_0^{1,p}(\Omega)\longrightarrow \R $ defined as
\begin{equation*}%\label{eq mat2}
D_v (\varphi ,\psi ):=\det \begin{pmatrix}
  D(P_{\ep}^+ - Q^+)(v) \varphi \quad & D(P_{\ep}^+ - Q^+)(v) \psi        \\[0.2em]
D(P_{\ep}^- - Q^-)(v)\varphi \quad & D(P_{\ep}^- - Q^-)(v)\psi          \\[0.2em]
\end{pmatrix}.
\end{equation*}

The previous considerations say $D_v (v^+ ,v^-)>0$. Aproximating $v^+$ and $v^-$ by functions  $\varphi$ and $\psi$ in $W$,  it follows that $(D(P_{\ep}^{+} - Q^+)(v), D(P_{\ep}^- - Q^-)(v)) |_{span\{\varphi,\psi \} }$ is invertible. Hence $(D\left(P^{+}_{\ep}- Q^{+}\right)(v),D\left(P^{-}_{\ep}- Q^{-}\right)(v)) \in B(W,\R^2)$ is surjective. This concludes the proof. $ \blacksquare$ 
\end{proof}

\medskip
Before proving the main theorem of this section we state a technical lemma that will help us to get an upper bound for the Morse Index of $u_{\ep}$.

%
%\begin{lemma}\label{lemma regularity}
%	Let $\bar{u}\in W_0^{1,p}(\Omega)$ be any solution of   \eqref{eq mean problem} then $\hat{u}\in W^{2,2}(\Omega)$
%\end{lemma}
%%
%\begin{proof}
%	Let $\bar{u}$ be any solution of \eqref{eq mean problem}, then by standard regularity results (see  \cite{tolksdorf1983dirichletproblem} and \cite{cingolani2003critical}) $u\in C^1(\bar{\Omega})$.\\
%	On the other hand, we can rewrite the equation \eqref{eq weak mean problem} as:
%\begin{equation}\label{eq elliptic formulation}
%\int_{\Omega} \sum_{i,j=1}^{N}a_{ij}\partial_i \hat{u}\partial_jv=\int_{\Omega} gv, \hspace{3mm} \forall v\in W_0^{1,p}
%\end{equation}
%	
%	With $a_{ij}:=\begin{cases}
%	1+\vert\hat{u}\vert^{p-2},\hspace{3mm}  \text{if}\, i=j\\
%	0, \hspace{8mm} \text{if}\, i\neq j
%	\end{cases}$ and with $g:=f(\hat{u})$.\\
%	Under these hypothesis is clear that the matrix $(a_{ij}(x))_{i,j}$ has Lipschitz components and that is uniformly elliptic with constant 1. Also is clear that $g\in L^2(\Omega)$, finally provided the regularity of $\Omega$ and using theorems $8.8$ and $8.12$ of \cite{trudinger1983elliptic} the result follows. $\blacksquare$
%\end{proof}   
%

\medskip
\begin{lemma}\label{lemma codimension}
	Let $\left(X, \| \cdot\|\right)$ be a Banach space and let $Y$ and $Z$ subspaces of $X$ such that $Y \leq Z$ and the codimension of $Y$ in $Z$ is $T\in \N$. Then $\overline{Y}$, the closure of $Y$ in $X$, has codimension at most $T$ in $\overline{Z}$.
\end{lemma}

\begin{proof}
Assume first that $Y$ has codimension one  in $Z$ so that there exists $z_0\in Z$, $z_0 \notin Y$, for which $Z=Y\oplus \R z_0$.

\medskip
If $z_0 \in \overline{Y}$, then $\overline{Z} = \overline{Y}$ and the result follows. Assume then that $z_0 \notin \overline{Y}$. Given $z\in \overline{Z}$, there exist a sequence $\{t_n z_0+y_n\}_{n\in\N}$, with $t_n\in \R$ and $y_n\in Y$ such that $t_nz_0 + y_n \to z$ and in particular $\{t_n z_0+y_n\}_{n\in\N}$ is bounded. 

\medskip
We claim that the sequence $\{t_n\}_{n\in\N}$ is bounded. If it were not the case, we could find a subsequence $\{t_{n_k}\}_{k\in\N}$ such that $\lim\limits_{k\to\infty} t_{n_k}=\infty$.

\medskip
Thus, there would exist $M>0$ such that for every  $k\in \N$,	\begin{equation*}
	\vert t_{n_k}\vert \left\Vert z_0 +\frac{1}{t_{n_k}}y_n\right\Vert_X \leq M
\end{equation*}
	
Implying that $\Vert z_0+\frac{1}{t_{n_k}}y_{n_k}\Vert_X\to 0$ as $k\to \infty$. This is a contradiction with the fact that we have assumed $z_0\notin \overline{Y}$ and so $\{t_n\}_{n\in \N}$ is bounded.

\medskip
Hence, up to subsequences, we may assume that $t_n \to t_0 \in \R$ and $y_n \to y_0 \in \overline{Y}$ as $n\to \infty$, so that $z= t_0 z_0+y_0$. 

\medskip
Since $z\in \overline{Z}$ is arbitrary, we conclude that $\overline{Z}=\R z_0 \oplus\overline{Y}$. The general case follows from the above discussion and an induction argument. This completes the proof of the lemma.  $ \blacksquare$
\end{proof}

\bigskip
	
\begin{proof}{\it Proof of Theorem \ref{Theorem2}.}	
First we prove {\it (i)}. Since $u\in \MM_{\ep}$,  hypothesis (f4), (specifically the properties of the functions proof $g_{u^+}(t)$ and $g_{u^-}(s)$ introduced in Lemma \ref{Mnonempty}) imply that for any $(\lambda,\eta)\in\R^2-(0,0)$,
\begin{equation*}
\begin{aligned}
D^2J_{\ep}(u)(\lambda u^+ +\eta u^-,\lambda u^+ +\eta u^- )&=\lambda^2D^2J_{\ep}(u)(u^+,u^+) +\eta^2D^2J_{\ep}(u)(u^-,u^-)\\
&<0.
\end{aligned}
\end{equation*}
 Therefore, $m(u_{\ep})\geq 2$.
\medskip
In virtue of Proposition \ref{regularityoftheNehariequator}, $u \in \mathcal{M}_{\ep}\cap W$, therefore Lemma \ref{lemma manifold} implies that $\mathcal{M}_{\ep}\cap W$ is a non-empty $C^1-$manifold of $W$. 
\medskip
Then,  for every $\varphi \in T_{u}\left(\mathcal{M}_{\ep}\cap W\right)$ there exists $\delta>0$ and a $C^1$ curve $c:(-\delta, \delta) \to W$ with $c(t)\in \MM_\varepsilon$ for every $t\in (-\delta, \delta)$ such that $c(0)=u$ and $c'(0)=\varphi$. Since $u$ is a local minimizer of $J_{\ep}$ on the $C^1$-manifold $\mathcal{M}_{\ep}\cap W$, then $\beta(t):=J_\varepsilon(c(t))$ has a local minimum at $0$. By the chain rule $\beta$ is a $C^1$ function and notice that since $u$ is a critical point of $J_\varepsilon$ then
\begin{equation*}
0 \leq \lim\limits_{t\to 0^+} \frac{\beta'(t)-\beta'(0)}{t}= \lim\limits_{t\to 0^+} \frac{\left(DJ(c(t))-DJ(u)\right)}{t}(c'(t))=D^2J(u)(\varphi,\varphi).
\end{equation*}

This implies that $D^2J(u)(\varphi,\varphi)\geq 0$ for every $\varphi \in T_{u}\left(\mathcal{M}_{\ep}\cap W\right)$.
\medskip
From Lemma \ref{lemma manifold}, $T_{u}\left(\mathcal{M}_{\ep}\cap W\right)$ has codimension 2 in $W$. Finally, using that $W$ is dense in $W_0^{1,p}(\Omega)$ and applying Lemma \ref{lemma codimension} to $T_{u}\left(\mathcal{M}_{\ep}\cap W\right)\leq W$ we get that the closure of $T_{u}\left(\mathcal{M}_{\ep}\cap W\right)$ has codimension at most two in $W^{1,p}_0(\Omega)$ and so $m_{\ep}(u)\leq 2$. 

\medskip

Next, we prove {\it (ii)}. Since $u\in \MM_{\ep}$ is a local minimizers, we conclude from part {\it (i)} in Lemma \ref{Nehari compactness} that $u$ has at least two nodal regions. Using part {\it (i)} from this result, $m_{\ep}(u)=2$ and directly from Lemma \ref{nodal regions} we conclude that $u$ has exactly two nodal regions. This completes the proof of the theorem. $\blacksquare$
\end{proof}

\medskip

\section{Proof of Theorem \ref{Theorem3}}\label{pass to limit}
Throughout the rest of the developments, we assume that $f$ satisfy (f1)-(f4) and we remind the reader that for $\ep \in \R$,
$$
J_{\ep}(v)= \int_{\Omega}\left(\frac{\ep^2}{2}|\nabla v|^2 + \frac{1}{p}|\nabla v|^p - F(v)\right)dx \quad \hbox{for} \quad  v\in W^{1,p}_0(\Omega).
$$

For any $\ep \in \R$, let $u_{\ep}\in \MM_{\ep}$ be the least energy nodal solution predicted by Theorem \ref{Theorem1}. From Lemma \ref{alpha=beta},
$$
\alpha_{\ep}= \min \limits_{v\in \MM_{\ep}} J_{\ep}(v) =\min \limits_{\substack{v\in W^{1,p}_0(\Omega),\\
v^+ ,\, v^-\neq 0}} \left(\max \limits_{t,\,s \geq 0} J_{\ep}(tv^+ + sv^-)\right).
$$

Also, a consequence of the proof of Lemma \ref{alpha=beta} is that 
\begin{equation}\label{alpha=betacorol}
J_{\ep}(u_{\ep}) = \max \limits_{t,s \geq 0} J_{\ep} (t u_{\ep}^+ + s u_{\ep}^-).
\end{equation}

Next two lemmas are concerned with the convergence of the sequences $\{\alpha_{\ep}\}_{\ep \geq 0}$ and $\{u_{\ep}\}_{\ep> 0}$

\begin{lemma}\label{energyuniformbound}
 The following assertions hold true:
\begin{itemize}
\item[(i)] the family $\{\alpha_{\ep}\}_{\ep}$ is strictly increasing in $\ep>0$, $\alpha_{\ep} \to \alpha_0$, as $\ep \to 0^+$ and

\item[(ii)] there exist $C>0$ and $\ep_0>0$ such that for any $\ep \in (0,\ep_0)$,
$$
\left( \frac{1}{p}- \frac{1}{m}\right)\|u_{\ep}\|_{W^{1,p}_0(\Omega)}^p\leq \alpha_{\ep}  \leq C.
$$
\end{itemize}
\end{lemma}

\begin{proof}
First, we prove {\it (i)}. Let $\ep_1,\ep_2\in [0,\infty)$ be arbitrary with $\ep_1<\ep_2$. Let also $\tau_{\varepsilon,v}>0$ denote the projection scalar defined in Lemma \ref{NNdiffeomSphere}, associated with the energy $J_{\ep}$ for the function $v$. Clearly $\tau_{\varepsilon_1,u_{\varepsilon_2}^+}u_{\varepsilon_2}^++\tau_{\varepsilon_1,u_{\varepsilon_2}^-} u_{\varepsilon_2}^- \in \mathcal{M}_{\varepsilon_1}$. Then,
\begin{equation*}
\alpha_{\ep_1} = J_{\varepsilon_1}(u_{\varepsilon_1})
\leq  J_{\varepsilon_1}(\tau_{\varepsilon_1,u_{\varepsilon_2}^+}u_{\varepsilon_2}^++\tau_{\varepsilon_1,u_{\varepsilon_2}^-} u_{\varepsilon_2}^-)
<  J_{\varepsilon_2}(\tau_{\varepsilon_1,u_{\varepsilon_2}^+}u_{\varepsilon_2}^++\tau_{\varepsilon_1,u_{\varepsilon_2}^-} u_{\varepsilon_2}^-).
\end{equation*}

From \eqref{alpha=betacorol},
$$
 J_{\varepsilon_2}(\tau_{\varepsilon_1,u_{\varepsilon_2}^+}u_{\varepsilon_2}^++\tau_{\varepsilon_1,u_{\varepsilon_2}^-} u_{\varepsilon_2}^-) \leq J_{\varepsilon_2}(u_{\varepsilon_2})
= \alpha_{\ep_2}
$$
and hence $\alpha_{\ep_1}< \alpha_{\ep_2}$.

\medskip

Next, we prove that $\alpha_\varepsilon \to \alpha_0$, as $\varepsilon\to 0^+$. The proof of Lemma \ref{NNdiffeomSphere} and the remark after it imply that  $\tau_{\ep,u_0^+}\to \tau_{0,u_0^+}$ and $\tau_{\ep,u_0^-}\to \tau_{0,u_0^-}$ as $\varepsilon\to 0^+$. Since $u_0\in\MM_0$, $\tau_{0,u_0^+}= \tau_{0,u_0^-} =1$ and hence the definition of $J_{\ep}$ yields that
\begin{equation*}
\alpha_0 <\alpha_{\ep}= J_\varepsilon(u_\varepsilon)\leq J_\varepsilon (t_{\ep,u_0}u_0^+ + s_{\ep,u_0}u_0^-)\longrightarrow J_{0}(u_0) = \alpha_0
\end{equation*}
as $\varepsilon\to 0^+$. This proves {\it (i)}.

\medskip
To prove {\it (ii)}, we notice that for any $\varepsilon_0>0$ fixed, from the previous discussion $\{\alpha_\varepsilon\,:\, \varepsilon\in (0,\varepsilon_0)\}$ is bounded and from Lemma \ref{QualitLemma} the conclusion follows. This completes the proof of the lemma.   $ \blacksquare$
\end{proof}

\medskip
%\blue{STILL NOT SURE IF INCLUDING THIS ARGUMENT.}

{\bf Remark:} More precise information about the asymptotics of $\alpha_{\ep}$ as $\ep \to 0^+$ can be obtain as follows. 
\medskip
Following the notations in the previous proof and using Lemma \ref{alpha=beta},
$$
\begin{aligned}
\alpha_{\ep} &\leq J_{\ep}(t_{\ep,u_0}u_0^+ + s_{\ep,u_0}u_0^-)\\
& =  J_{0}(t_{\ep,u_0}u_0^++ s_{\ep,u_0}u_0^-) + \frac{\ep^2 t_{\ep,u_0}^2}{2} \int_{\Omega}|\nabla u_0^+|^2 dx + \frac{\ep^2 s_{\ep,u_0}^2}{2} \int_{\Omega}|\nabla u_0^-|^2dx.
\end{aligned}
$$
Since $u_0 \in \MM_0$, from \eqref{alpha=betacorol} for $\ep=0$,
$$
J_{0}(t_{\ep,u_0}u_0^++ s_{\ep,u_0}u_0^-) \leq  J_0 (u_0) = \alpha_0.
$$
Therefore,
\begin{equation}\label{crucialestimateI}
\alpha_{\ep} \leq \alpha_{0} +  \frac{\ep^2 t_{\ep,u_0}^2}{2} \int_{\Omega}|\nabla u_0^+|^2 dx + \frac{\ep^2 s_{\ep,u_0}^2}{2} \int_{\Omega}|\nabla u_0^-|^2dx.
\end{equation}
The convergence of $\{t_{\ep,u_0}\}_{\ep>0}$ and $\{s_{\ep,u_0}\}_{\ep>0}$, as $\ep\to 0^+$, implies that for some $\ep_0 >0$ small enough we have that $\{t_{\ep,u_0}\}_{\ep\in (0,\ep_0)}$ and $\{s_{\ep,u_0}\}_{\ep\in (0,\ep_0)}$ are bounded. 
\medskip
Thus, fixing $\ep_0>0$ as above and using \eqref{crucialestimateI}, we find a constant $C>0$ such that for any $\ep \in (0,\ep_0)$,
\begin{equation*}%\label{alphaepsmalleralphazero}
\alpha_0 \leq \alpha_{\ep} \leq \alpha_0 + C \ep^2. 
\end{equation*}

\begin{lemma}\label{convergence least energies}
	There is a sequence $\{u_{\varepsilon_n}\}_{n\in \N}$ converging strongly to some $\red{{\rm u}_0} \in \MM_0$ which is a least energy nodal solution for \eqref{eq mean problem} with $\varepsilon=0$.
\end{lemma}
%\blue{Maybe this argument can be cut-off, since it is almost the same, line by line, as the argument in Lemma \ref{Nehari compactness}.}

\begin{proof}
This proof follows essentially the arguments of the proof of Lemma \ref{Nehari compactness}. From Lemma \ref{energyuniformbound}, we find that $\{u_{\ep}\}_{\ep>0}$ is bounded in $W^{1,p}_0(\Omega)$. The reflexivity of $W^{1,p}_0(\Omega)$ and the compactness of the Sobolev embedding yield the existence of ${\rm u}_0 \in W^{1,p}_0(\Omega)$ such that, up to a subsequence $u_{\varepsilon_n} \rightharpoonup {\rm u}_0$ weakly in $W^{1,p}_0(\Omega)$ and $u_{\varepsilon_n}^+\to {\rm u}_0^+$, $u_{\varepsilon_n}^- \to {\rm u}_0^-$ strongly in $L^r(\Omega)$ for any $r\in [1,p^*)$.

\medskip
Vainberg's lemma (see \cite{VAINBERG64}) and condition (f1) implies that as $n \to \infty$,
\begin{equation*}
\int_{\Omega}F(u_{\varepsilon_n}^\pm) dx\to \int_{\Omega} F({\rm u}_0^\pm) dx,
\end{equation*}
\begin{equation*}
\int_{\Omega}f(u_{\varepsilon_n}^\pm)u_{\varepsilon_n}^\pm dx \to \int_{\Omega} f(\rm u_0^\pm){\rm u}_0^\pm dx.
\end{equation*}

For $n\in \N$, denote $v_n:=u_{\varepsilon_n}$ and set $v_0= {\rm u}_0$. By taking further subsequences if necessary, we may assume that $\{ \|v_n^+ \| \} _n$ and $\{ \|v_n^- \| \} _n$ converge in $\Bbb R$ so that
\begin{equation}\label{subsequenceeps}
\|v_0^{+} \| _{W^{1,p}_0(\Omega)} \leq \lim\limits_{n\rightarrow \infty} \|v_n^+ \| _{W^{1,p}_0(\Omega)} \ \text{ and } \ \|v_0^{-} \| _{W^{1,p}_0(\Omega)} \leq \lim\limits_{n\rightarrow \infty} \|v_n^- \| _{W^{1,p}_0(\Omega)} .
\end{equation}

Also, observe that $v_0^+ \geq 0$ and $v_0^- \leq 0$ a.e. in $\Omega$. 
Since $v_{n}\in \MM_{\varepsilon_n}$, 
\begin{equation*}%\label{inMMep}
DJ_{\varepsilon_n}(v_n^{\pm})v_n^{\pm} =0 \quad \hbox{for every} \quad n\in \mathbb{N}
\end{equation*}
and using Lemma \ref{QualitLemma} we find that
\begin{equation*}
\begin{aligned}
\int_{\Omega}f(v_0^{\pm})v_0^{\pm} dx=&\lim \limits_{n\to \infty}\int_{\Omega}f(v_n^{\pm})v_n^{\pm}dx \\ 
\geq& \lim \limits_{n\to \infty}\int_{\Omega}|\nabla v_n^{\pm}|^p dx\\\geq&  \rho^p
\end{aligned}
\end{equation*}
and hence $v_0^{\pm} \neq 0$ in $W^{1,p}_0(\Omega)$.\\

Next we prove that $v_n \to v_0$ strongly in $W^{1,p}_0(\Omega)$. It suffices to prove both equalities in \eqref{subsequenceeps} hold. To this end, let us argue by contradiction: assume either 
\begin{equation}\label{subsequence4ep}
\|v_0^{+} \| _{W^{1,p}_0(\Omega)} < \lim\limits_{n\rightarrow \infty} \|v^{+}_n \| _{W^{1,p}_0(\Omega)} \ \text{ or } \ \|v_0^{-} \| _{W^{1,p}_0(\Omega)} < \lim\limits_{n\rightarrow \infty} \|v^{-}_n \| _{W^{1,p}_0(\Omega)} .
\end{equation}

Set $a:=\tau_{0,v_0^+}$ and $b:=\tau_{0,v_0^-}$ the unique positive numbers such that $a v_0^+, bv_0^- \in \NN_{0}$. Consequently, $av_0^+ + bv_0^- \in \MM_{0}$ and then
$$
\begin{aligned}
\alpha_0 
\leq & J_{0}(av_0^+ + bv_0^-)=J_{0}(av_0^+ ) + J_{0} (bv_0^-)\\
< & \lim_{n\to \infty} J_{\ep_n}( a v_n^+)+\lim \limits_{n\to \infty}J_{\ep_n}( bv_n^-) \ \ \text{ (from }  \eqref{subsequence4ep}) \\
\leq & \lim_{n\to \infty} J_{\ep_n}(  v_n^+)+\lim \limits_{n\to \infty}J_{\ep_n}( v_n^-)  \ \ \text{ (since } v_n^+ , v_n^- \in \NN_{\ep_n}) \\
=& \lim_{n\to \infty} J_{\ep_n}(  v_n^+ + v_n^-) =\lim_{n\to \infty} \alpha_{\varepsilon_n}=\alpha_0 \quad (\text{from Lemma \ref{energyuniformbound}}).
\end{aligned}
$$

This contradiction implies that both equalities hold in \eqref{subsequenceeps}. Since $W_0^{1,p}(\Omega)$ is uniformly convex (see Theorem 2.6 in \cite{AdamsFournier}) the convergence of $\{v_n\}_{n\in \N}$ to $v_0$ is strong.

\medskip
Finally, using that $v_n \to v_0$ strongly in $W^{1,p}_0(\Omega)$, we conclude that $DJ_{0}(v_0^{\pm})v_0^{\pm}=0$ and $J_0(v_0)=\alpha _{0}$. This completes the proof. $\blacksquare$
\end{proof}  

%\blue{The next proof was rewritten almost entirely. Please review it to check for errors.}

\begin{proof}{\it Proof of Theorem \ref{Theorem3}.} 
% Notice that if $\left\{\Vert DJ(v_n)\Vert_{\left(W_0^{1,p}(\Omega)\right)^*}\right\}_{n\in \N}$ is not bounded by below then the Palais-Smale condition implies that $\{v_n\}_{n\in \N}$ converges, up to a subsequence, strongly to $v$ leading to a contradiction and proving our claim. In other case there exists $K>0$ such that $\Vert DJ(v_n)\Vert_{\left(W_0^{1,p}(\Omega)\right)^*}\geq K$ for every $n\in \N$. Since $DJ$ is uniformly continuous in bounded sets we can take an square $D\subset \R^2_{++}$ containing $(1,1)$ and $\delta>0$ such that $B_\delta(v_n)\subset  \{tv_n^++sv_n^-|(t,v)\in \overline{D}\} \subset B_r(u)$ and $DJ(w)\geq \frac{K}{2}$ for every $w\in B_\delta(v_n)$ for every $n\in \N$. Taking $n$ big enough such that $J(v_n)<J(u)+ \frac{\delta K}{8}$ we can proceed as in Lemma 10 and find a deformation $\Lambda$ such that 
%$$\max_{(t,s)\in \overline{D}}J(\Lambda(1,tv_n^++sv_n^-)<$$  \\
Let $\{u_{\ep_n}\}_{n\in \N}$ and ${\rm u}_0$ be as in Lemma \ref{convergence least energies} so that $u_{\ep_n} \in \MM_{\ep_n}\cap W^{2,p}(\Omega)$ with $J_{\ep_n}=\alpha_{\ep_n}$, ${\rm u}_0 \in \MM_0$ with $J_0({\rm u}_0)=\alpha_0$ and $u_{\ep_n} \to {\rm u}_0$ strongly in $W^{1,p}_0(\Omega)$.

\medskip
Recall that 
\begin{multline}\label{2ndDerivJ0}
DJ^2_{0}({\rm u}_0)(\varphi, \varphi) = \int_{\Omega} |\nabla {\rm u}_0|^{p-2}|\nabla \varphi|^2dx + \int_{\Omega}(p-2)|\nabla {\rm u}_0|^{p-4}(\nabla {\rm u}_0 \cdot \nabla \varphi)^2 dx \\
- \int_{\Omega} f'({\rm u}_0)\varphi^2 dx
\end{multline}
for every $\varphi \in W^{1,p}_0(\Omega)$.

\medskip
Since ${\rm u}_0^+, {\rm u}_0^+ \neq 0$, we test \eqref{2ndDerivJ0} against $\varphi={\rm u}_0^{\pm}$ and use hypothesis (f4) to conclude that $m_0({\rm u}_0)\geq 2$. Thus, it suffices to show that for some subspace $V$ of $W^{1,p}_0(\Omega)$ with codimension two,
\begin{equation}\label{MorseIndexgreater2}
DJ^2_{0}({\rm u}_0)(\varphi, \varphi) \geq 0 \quad \hbox{for every}\quad \varphi \in  V.
\end{equation}

\medskip
We proceed as follows using the same notations as in the proof of Theorem \ref{Theorem2} in Section 4. 

\medskip
For every $n\in \N$, consider the functions $\Phi_{\ep_n}^{\pm}:W^{1,p}_0(\Omega) \to \R$ defined by
 \begin{align*}
 \Phi^\pm_{\varepsilon_n}(\varphi)= \int_{\Omega} ( \varepsilon_n^2 \nabla u_{\varepsilon_n}^\pm + (p-1)|\nabla u_{\varepsilon_n}|^{p-2}\nabla u_{\varepsilon_n}^\pm) \cdot \nabla \varphi dx -\int_{\Omega}f'(u_{\varepsilon_n}^\pm)u_{\varepsilon_n}^\pm \varphi dx
 \end{align*}
for $\varphi \in W^{1,p}_0(\Omega)$.

\medskip
From Lemma \ref{dif}, for every $\varphi \in W$, 
 \begin{multline*}
\Phi^+_{\varepsilon_n}(\varphi):=\int_{\{ u_{\varepsilon_n}>0\}}(- \varepsilon_n^2 \Delta u_{\varepsilon_n} -\Delta_p u_{\varepsilon_n})\varphi dx \\
+ \int_{\Omega} ( \varepsilon_n^2 \nabla u_{\varepsilon_n}^+ + (p-1)|\nabla u_{\varepsilon_n}|^{p-2}\nabla u_{\varepsilon_n}^+) \cdot \nabla \varphi dx\\
-\int_{\Omega}f'(u_{\varepsilon_n}^+)u_{\varepsilon_n}^+ \varphi dx -\int_{\Omega} f(u_{\varepsilon_n}^+)\varphi dx
\end{multline*}
and
\begin{multline*}
\Phi^-_{\varepsilon_n}(\varphi):=\int_{\{ u_{\varepsilon_n}<0\}}(- \varepsilon_n^2 \Delta u_{\varepsilon_n} -\Delta_p u_{\varepsilon_n})\varphi dx \\
+ \int_{\Omega} ( \varepsilon_n^2 \nabla u_{\varepsilon_n}^- + (p-1)|\nabla u_{\varepsilon_n}|^{p-2}\nabla u_{\varepsilon_n}^-) \cdot \nabla \varphi dx\\
-\int_{\Omega}f'(u_{\varepsilon_n}^-)u_{\varepsilon_n}^- \varphi dx -\int_{\Omega} f(u_{\varepsilon_n}^-)\varphi dx.
\end{multline*}

We conclude then that the tangent space $T_{\ep_n}(\MM_{\ep_n}\cap W)$ in $W$, consists on functions $\varphi \in W$ such that 
$$
\Phi^{+}_{\ep_n}(\varphi)=\Phi^{+}_{\ep_n}(\varphi)=0.
$$ and from the proof of part {\it (i)} in Theorem \ref{Theorem2}, for every $\varphi \in T_{u_{\varepsilon_n}}\left(\MM_{\varepsilon_n}\cap W\right)$, $D^2J_{\varepsilon_n}(u_{\varepsilon_n})(\varphi,\varphi)\geq 0$. 

\medskip
Let $n\in \N$ be arbitrary, but fixed and set
$$
V_{n}:=\overline{\{\varphi \in W\,:\, \Phi^{\pm}_{\varepsilon_n}(\varphi)=0\}}^{W_0^{1,p}(\Omega)}.
$$

We claim that $V_n = E_n$, where  
$$
E_n:={\{\varphi \in W_0^{1,p}(\Omega)\,:\, \Phi^{\pm}_{\varepsilon_n}(\varphi)=0\}}.
$$

To prove the claim, we notice first that $V_n \subseteq E_n$. Also, from Lemma \ref{lemma codimension}, ${\rm codim}V_n  \leq 2$ in $W^{1,p}_0(\Omega)$.  

\medskip
In order to prove the reverse inclusion, we prove that ${\rm codim}E_n =2$. Arguing in the same fashion as in the proof of {\it (ii)} and {\it (iii)} in Lemma \ref{furtherpropofNN}, we find that $$
\Phi^{+}_{\ep_n}(u^+_{\ep_n}),\Phi^{-}_{\ep_n}(u^-_{\ep_n}) <0 \quad \hbox{and} \quad \Phi^{+}_{\ep_n}(u_{\ep_n}^{-})= \Phi^{-}_{\ep_n}(u_{\ep_n}^{+})=0.
$$  

Therefore, given any $\varphi \in W^{1,p}_0(\Omega)$, we can set 
\begin{equation}\label{choiceab}
a:= -\frac{\Phi^+_{\ep_n}(\varphi)}{\Phi^+_{\ep_n}(u_{\ep_n}^+)} \quad \hbox{and} \quad b:= -\frac{\Phi^-_{\ep_n}(\varphi)}{\Phi^-_{\ep_n}(u_{\ep_n}^-)}
\end{equation}
so that $w_{n}:= \varphi - au^{+}_{\ep_n} - b u^-_{\ep_n}$ satisfies $\Phi^+_{\ep_n}(w_{n})= \Phi^-_{\ep_n}(w_{n})=0$. This proves that given any $\varphi \in W^{1,p}_0(\Omega)$, there exist unique $a,b \in \R$ and $w_{\ep_n}\in E_n$ such that $\varphi = a u^+_{\ep_n} + b u_{\ep_n}^- + w_{\ep_n}$, i.e. $$
W^{1,p}_0(\Omega)= \R{u_{\ep_n}^+}\oplus \R {u_{\ep_n}^-} \oplus E_n.
$$

Even more, from \eqref{choiceab} it follows that for any $\varphi \in E_n$, $a=b=0$.

\medskip
Since $V_n, E_n$ are closed subspaces of $W^{1,p}_0(\Omega)$, $V_n \subset E_n$ and
$$
{\rm codim} E_n= 2 \geq {\rm codim} V_n,
$$ 
we conclude that $V_n=E_n$ and this proves the claim.

\medskip
We finish the proof of Theorem \ref{Theorem3} by a limiting process. Define 
 \begin{align*}
\Phi^\pm_{0}(\varphi):= \int_{\Omega}  (p-1)|\nabla u_{0}|^{p-2}\nabla u_{0}^\pm \cdot \nabla \varphi dx -\int_{\Omega}f'(u_{0}^\pm)u_{0}^+ \varphi dx
\end{align*}
for $\varphi \in W^{1,p}_0(\Omega)$ and set
$$ 
V:=\{\varphi\in W_0^{1,p}(\Omega)\,:\, \Phi^{\pm}_0(\varphi)=0\}.
$$

Proceeding as above, it follows that ${\rm codim}V =2$ in $W^{1,p}_0(\Omega)$. We prove next that inequality \eqref{MorseIndexgreater2} holds true for this choice of $V$. Let $\varphi \in V$ be arbitrary, but fixed. Using \eqref{choiceab}, we set
$$
a_n:=-\dfrac{\Phi^+_{\varepsilon_n}(\varphi)}{\Phi^+_{\varepsilon_n}(u_{\varepsilon_n}^+)}
\quad \hbox{and} \quad b_n:=-\dfrac{\Phi^-_{\varepsilon_n}(\varphi)}{\Phi^-_{\varepsilon_n}(u_{\varepsilon_n}^-)}
$$
so that $w_n:=\varphi - a_{n}u_{\ep_n}^+ - b_{n}u_{\ep_n}^- \in V_{n}$.

\medskip
Since $u_{{\varepsilon_n}} \to {\rm u}_0$ strongly  in $W^{1,p}_0(\Omega)$, as $n \to \infty$ and $\Phi^\pm_0({\rm u}_0^\pm)<0$, 
$$
a_{n}\to -\dfrac{\Phi^+_0(\varphi)}{\Phi^+_0({\rm u}_0^+)}=0 \quad  \hbox{and} \quad b_{n}\to -\dfrac{\Phi^-_0(\varphi)}{\Phi^-_0({\rm u}_0^-)}=0.
$$

Consequently, as $n \to \infty$, $\varphi + a_{n}u_{\ep_n}^+ + b_{n}u_{\ep_n}^- \to \varphi$ strongly in $W_0^{1,p}(\Omega)$. On the other hand, notice that
\begin{align*}
\left|D^2J_{\varepsilon_n}(u_{\ep_n})(w_n,w_n)-D^2J_{0}({\rm u}_0)(\varphi,\varphi)\right|\leq &\\
 &\left|\left(D^2J_{0}(u_{\ep_n})-D^2J_{0}({\rm u}_0)\right)(w_n,w_n)\right|\\
 &+\left|D^2J_{0}({\rm u}_0)(w_n,w_n)-D^2J_{0}({\rm u}_0)(\varphi,\varphi)\right|\\
 &+\ep_n^2\int_{\Omega}|\nabla w_n|^2dx
\end{align*}

Since $u_{\ep_n}\to {\rm u}_0$ strongly in $W^{1,p}_0(\Omega)$ and $a_n,b_n \to 0$, as $n\to \infty$, we conclude that $D^2J_{\varepsilon_n}(u_{\ep_n})(w_n,w_n)\to D^2J_{0}({\rm u}_0)(\varphi,\varphi)$ as $n\to \infty$. Using that $w_n \in V_n$ and that $D^2 J_{\ep_n}(u_{\varepsilon_n})(\cdot,\cdot)|_{V_{n}} \geq 0$, we conclude that $D^2J_{0}({\rm u}_{0})(\varphi,\varphi) \geq 0$. Since $\varphi \in V$ is arbitrary and $V$ has codimension two in $W^{1,p}_0(\Omega)$, we find that $m_0(u_0) \geq 2$.

\medskip
Therefore, $m_0({\rm u}_0)=2$. Finally, Lemma  \ref{nodal regions} implies that ${\rm u}_0$ has exactly two nodal domains. This completes the proof of the theorem.  $\blacksquare$
\end{proof}

\section{Proof of Theorem \ref{Theorem4}.}
We begin this section with some comments that are crucial for subsequent developments.

\begin{itemize}
\item[(a)] From Lemma \ref{Deformation}, for every $\varepsilon \geq 0$, any local minimizer of $J_\varepsilon |_{\MM_\varepsilon}$ is a critical point of $J_\varepsilon$.

		\item[(b)]  In particular, for $\varepsilon=0$, if $u \in \MM_0$ is a local minimizers of $J_0|_{\MM_0}$, which is in addition isolated critical point of $J_0$,  then it must be a strict local minimizer of $J_0|_{\MM _0}$.
		
		\item [(c)] Since $p >2$ and $f \in C^1(\R)$, the energy $J_0:W^{1,p}_0(\Omega) \to \R$ and it derivative $DJ_0:W^{1,p}_0(\Omega) \to (W^{1,p}_0(\Omega))^*$ are uniformly continuous on bounded subsets of $W^{1,p}_0(\Omega)$. 
\end{itemize}

\medskip
	
\begin{proof}{\it Proof of Theorem \ref{Theorem4}.} 
For the sake of clarity, we present this proof splited in five steps. Let $f$ satisfy (f1)-(f4). 

\bigskip
%\blue{Claim 1 and 2 are interchanged, because the current Claim 1 is generic and depends only on the hypothesis made on $f$. }
%
%\medskip
%\blue{Please check also that the references to the previous Lemmas are correctly cited!!}

\textbf{Claim 1.} Let $R>0$ and $\varepsilon'>0$ be arbitrary, but fixed. There exists $K>0$, depending only on $R$ such that for every $\varepsilon \in [0,\varepsilon']$ and every $v \in B_R(0)\cap \NN_{\varepsilon'}$, 
$$
|\tau_{\varepsilon, v}-\tau_{0,v}|\leq K\varepsilon^2,
$$
where $\tau_{\ep,v},\tau_{0,v}>0$ are such that $\tau_{\ep,v}v\in \NN_{\ep}$ and $\tau_{0,v}v\in \NN_0$ are the scalars described in Lemma \ref{Mnonempty}.

\medskip		
\textit{Proof Claim 1.} First we remark that the implicit function theorem applied to the function $\xi$, introduced in the proof of Lemma \ref{NNdiffeomSphere}, yields that
the function 
$$
\R \times W^{1,p}_0(\Omega)-\{0\}\ni (\ep,v) \mapsto \tau_{\varepsilon, v} \in (0,\infty)$$ 
is continuously differentiable with 
\begin{equation}\label{derivativetau}
\begin{aligned}
\frac{\partial \tau_{\ep,v}}{\partial \ep} = &  - \frac{\partial \xi}{\partial \ep}(\ep, \tau_{\ep,v},v)\left[\frac{\partial \xi}{\partial \tau}(\ep, \tau_{\ep,v},v)\right]^{-1}\\
= & -\dfrac{2\ep \int_{\Omega}|\nabla \tau_{\ep,v}v|^2dx}{ (p-2)\ep^2 \tau_{\ep,v}\int_{\Omega} |\nabla v|^2 dx  +  \int_{\Omega}\left(f'(\tau_{\ep,v}v)\tau_{\ep,v} v^2- (p-1)f(\tau_{\ep,v} v)v \right)dx}.
\end{aligned}
\end{equation}

Using hypothesis (f4), we conclude that both integrals on the denominator of \eqref{derivative t} are positive and therefore for any fixed $v\in W^{1,p}_0(\Omega)$, $v\neq 0$, the function $[0,\infty)\ni \ep \mapsto \tau_{\ep,v}>0$ is a strictly decreasing function in $\ep$.

\medskip
Let $\ep \in [0,\ep']$ and $v\in \NN_{\ep'}\cap B_R(0)$. Part {\it (ii)} in Lemma \ref{Mnonempty} implies that $DJ_\varepsilon(v)(v)\leq 0$ and part {\it (iii)} from the same Lemma yields $\tau_{\varepsilon, v}\in (0,1]$.

\medskip
Using the Mean Value Theorem and \eqref{derivativetau}, we find $\eta \in (0,\varepsilon)$ such that
\begin{equation}\label{derivative t}
		|\tau_{\varepsilon,v}-\tau_{0,v}|=\dfrac{2\eta\ep \int_{\Omega}|\nabla \tau_{\eta,v}v|^2dx}{ (p-2)\ep^2 \tau_{\eta,v}\int_{\Omega} |\nabla v|^2 dx  +  \int_{\Omega}\left(f'(\tau_{\eta,v}v)\tau_{\eta,v} v^2- (p-1)f(\tau_{\eta,v} v)v \right)dx}
\end{equation}
	
Observe that 
\begin{equation}\label{estimatelate}
2\eta\ep \int_{\Omega}|\nabla \tau_{\eta,v}v|^2 \leq 2 \ep^2 R^2.
\end{equation}

 Also, We rewrite the second term of the denominator in \eqref{derivative t} as
	\begin{equation*}
 \frac{1}{\tau_{\eta,v}}\int_{\Omega}\left(f'(\tau_{\eta,v}v)(\tau_{\eta,v} v)^2- (p-1)f(\tau_{\eta,v} v)\tau_{\eta,v}v, \right)dx. 
\end{equation*}

From part {\it (iii)} in Lemma \ref{QualitLemma}, 
$$
\bigcup_{\hat{\ep}\geq 0} \NN_{\hat{\ep}} \subset \{w\in W^{1,p}_0(\Omega)\,:\, \|w\|_{L^q(\Omega)} \geq \rho\}
$$
and since $0<\tau_{\eta,v}\leq \tau_{\ep,v}\leq 1$ and $\tau_{\eta,v}v\in \bigcup_{\hat{\ep}\geq 0} \NN_{\hat{\ep}}$,  
\begin{equation}\label{problem min}
\frac{1}{\tau_{\eta,v}}\int_{\Omega}\left(f'(\tau_{\eta,v}v)(\tau_{\eta,v} v)^2- (p-1)f(\tau_{\eta,v} v)\tau_{\eta,v}v, \right)dx \geq M,
\end{equation}
where 	
\begin{equation}\label{min problem}
M=\inf\left\{I(w):= \int_{\Omega}f'(w)w^2- (p-1)f(w)wdx\Big|\,\, w\in B_R(0), \|w\|_{L^q(\Omega)} \geq \rho \right\}.
	\end{equation}

%\blue{Observe the new definition of the set, the use of the variable $v$ created confusion in the proof so I changed to $w$.}
	
Next, we prove that $M>0$. Arguing by contradiction, assume that $\{w_n\}_{n\in \N} \subset B_R(0)$ with $\|w_n\|_{L^q(\Omega)} \geq \rho$ is a minimizing sequence for \eqref{min problem} such that $\lim\limits_{n\to \infty} I(w_n)=0$.

\medskip
The compactness of the Sobolev embedding $W^{1,p}_0(\Omega) \hookrightarrow L^q(\Omega)$ implies that there exists $w\in W_0^{1,p}(\Omega)$ such that, up to a subsequence, $w_n \to w$ strongly in $L^q(\Omega)$. In particular, $\|w\|_{L^q(\Omega)} \geq \rho >0$

\medskip
On the other hand, Vainberg's Lemma implies that
$$
I(w)= \int_{\Omega}f'(w)w^2- (p-1)f(w)wdx=0.
$$

Using again hypothesis (f4), 
$$
f'(w)w^2- (p-1)f(w)w=0 \quad \hbox{a.e. in} \quad \Omega 
$$
and consequently $w=0$ a.e. in $\Omega$. This is clearly a contradiction and hence $M>0$. 

\medskip
From \eqref{derivative t},\eqref{estimatelate} and \eqref{problem min}, we conclude that 
$$
|\tau_{\ep,v}- \tau_{0,v}| \leq \frac{2R}{M}\ep^2
$$
and this proves the claim.

\bigskip
Next, let $u\in \MM_0$ be a strict local minimizer of $J_{0}|_{\MM_0}$ so that there exists $r>0$ with 
$$
J_{0}(u) < J_0(v) \quad \hbox{for every} \quad v\in \MM_0\cap B_r(u), \quad v\neq u.
$$

\textbf{Claim 2.} For every $s_1,s_2\in(0,r)$ with $s_1<s_2$,
\begin{equation*}%\label{strict min}
J_0(u)<\inf \limits_{\substack{v \in \MM_0,\\ s_1 \leq \Vert v-u\Vert_{W^{1,p}_0(\Omega)} \leq s_2}} J_0(v).	
\end{equation*}

\textit{Proof of Claim 2.} 
%\blue{Please check that the references to the previous Lemmas are correctly cited!!}

Proceeding by contradiction, assume the existence of a sequence $\{v_n\}_{n\in \N}\subset\MM_0$ such that $\Vert v_n-u\Vert_{W_0^{1,p}(\Omega)}\in[s_1,s_2]$ and $J_0(v_n)\to J_0(u)$, as $n\to \infty$. 

\medskip
Since $\{v_n\}_{n\in \N}$ is bounded in $W^{1,p}_0(\Omega)$, there exists a subsequence, which we denote the same, and there exists $v\in \overline{B_{s_2}(u)}$ such that $v_n \rightharpoonup v$ weakly in $W^{1,p}_0(\Omega)$. 

\medskip
Using that $u\in \MM_0$ is a least energy nodal solution for $(\mathcal{P}_0)$ and part {\it (ii)} in  Lemma \ref{Nehari compactness},  we conclude that $v_n\to v$ strongly in $W_0^{1,p}(\Omega)$. Therefore, $v \in \MM_0$, $\Vert v-u\Vert_{W_0^{1,p}(\Omega)}\in [s_1,s_2]$ and $J_0(v)=J_0(u)$. This contradicts the fact $u$ is a strict minimizer in $B_r(u)$. This proves the claim.

\bigskip

\textbf{Claim 3.} For every $s\in (0,\frac{r}{2})$, there exists $\ep_0>0$ such that for every $\varepsilon\in (0,\ep_0)$,  
		\begin{equation}\label{approximating inf}
		\inf \limits_{v\in B_s(u)\cap \MM_\varepsilon} J_\varepsilon(v)
		\end{equation}		
is well defined and it is attained.

\medskip
\textit{Proof Claim 3.} Let $s\in (0,\frac{r}{2})$ be arbitrary, but fixed. Since the function 
\begin{equation}\label{functiontau}
\R \times W^{1,p}_0(\Omega)-\{0\}\ni (\ep,v) \mapsto \tau_{\varepsilon, v} \in (0,\infty)
\end{equation} 
is continuously differentiable, there exists $\hat{\ep}_0>0$ and $s'\in (0,s)$ such that for every $\varepsilon\in [0,\hat{\ep}_0)$ and every $v\in B_{s'}(u)$,
\begin{equation}\label{projection ball}
\Vert \tau_{\varepsilon, v^+}v^++ \tau_{\varepsilon, v^-}v^--u\Vert_{W^{1,p}_0(\Omega)}<s'.
\end{equation}

The continuity of the mappings 
$$
W^{1,p}_0(\Omega)\ni v\to v^{\pm} \in W^{1,p}_0(\Omega)
$$
allows to assume further that $s'\in (0,s)$ is such that  for every $v\in B_{s'}(u)$, $v^{\pm}\neq 0$. In particular, $B_{s'}(u)\cap \MM_{\ep}$ is non-empty and hence the infimum in \eqref{approximating inf} is well defined.

\medskip
Next, we prove that the infimum in \eqref{approximating inf} is attained. Fix $s_1 \in (0,s')$ and set $s_2=\frac{r}{2}$. Claim 1, guarantees the existence of $C_{s_1}>0$ such that 
\begin{equation}\label{Cs1}
J_0(u)=\inf \limits_{\substack{v \in \MM_0,\\ s_1 \leq \Vert v-u\Vert_{W^{1,p}_0(\Omega)} \leq s_2}} J_0(v) \,-\, C_{s_1}.	
\end{equation}

For any given $\varepsilon>0$, set $w_\varepsilon:=\tau_{\varepsilon,u^+}  u^++\tau_{\varepsilon, u^-} u^-$. Observe that $w_{\ep}\in \MM_{\ep}$. The continuity in $\ep \in \R$ of the function in \eqref{functiontau}, fixing $v= u^+$ and $v= u^-$, yields the existence of $\hat{\ep}_1>0$ such that for $\varepsilon\in (0,\hat{\ep}_1)$, $w_\varepsilon \in B_{s_1}(u)$ and since and $J_\varepsilon (w_\varepsilon)\to J_0(u)$ as $\varepsilon \to 0^+$, we may choose $\hat{\ep}_1$ smaller if necessary, so that
\begin{equation}\label{first ineq}
J_\varepsilon(w_\varepsilon)<J_0(u)+\frac{C_{s_1}}{4}.
\end{equation}

Finally, we claim that there exists $\hat{\ep}_2>0$ such that for every $\varepsilon \in (0,\hat{\ep}_2)$ and for every $v \in \left(B_{s}(u)\setminus\overline{B_{s'}(u)}\right)\cap \MM_\ep$,
\begin{equation}\label{second ineq}
 \hbox{and} \quad 
	J_0(\tau_{0,v^+}v^++\tau_{0,v^-}v^-)<J_\varepsilon(v)+\frac{C_{s_1}}{4}.	  
\end{equation}

Let us assume for the moment that this last claim holds true. We finish the proof of Claim 3 proceeding as follows. 

\medskip
Set $\ep_0:=\min\{\hat{\ep}_0, \hat{\ep}_1,\hat{\ep}_2\}$. Let $\varepsilon \in (0,\ep_0)$ and $v\in \left(B_{s}(u)\setminus\overline{B_{s'}(u)}\right)\cap \MM_\ep$ be arbitrary, but fixed.

\medskip
Observe that $\tau_{0, v^+}v^++\tau_{0,v^-}v^- \in B_{s_2}(u)\setminus\overline{B_{s_1}(u)}$. Combining \eqref{first ineq}, \eqref{second ineq} and \eqref{Cs1}, the choice of $C_{s_1}$, 
$$
\begin{aligned}
J_{\ep}(w_{\ep}) < &  J_0(u) + \frac{C_{s_1}}{4} \qquad (\hbox{from} \, \eqref{first ineq}) \\
= & \inf \limits_{\substack{v \in \MM_0,\\ s_1 \leq \Vert v-u\Vert_{W^{1,p}_0(\Omega)} \leq s_2}} J_0(v) \,-\, \frac{3C_{s_1}}{4} \qquad (\hbox{from} \, \eqref{Cs1})\\
\leq & J_{0}(\tau_{0,v^+}v^+ + \tau_{0,v^-}v^-) - \frac{3C_{s_1}}{4} \\
\leq & J_{\ep}(v) - \frac{C_{s_1}}{2} \qquad (\hbox{from} \, \eqref{second ineq}).
\end{aligned}
$$

\medskip
We conclude that for every $\varepsilon \in (0,\ep_0)$ and every $v\in \left(B_{s}(u)\setminus\overline{B_{s'}(u)}\right)\cap \MM_\ep$, 
\begin{equation}\label{cruciale}
J_\varepsilon(w_\varepsilon)<J_\varepsilon(v)-\frac{C_{s_1}}{2}.
\end{equation}

Let $\varepsilon\in (0,\ep_0)$ and let $\{v_n\}_{n\in \N} \subset \MM_{\ep}\cap \overline{B_s(u)}$ be a minimizing sequence for \eqref{approximating inf}. We may assume with no loss of generality that $v_n \rightharpoonup v$ weakly  in $W^{1,p}_0(\Omega)$ for some $v$ satisfying that $\Vert v-u\Vert_{W^{1,p}_0(\Omega)} \leq s$. 

\medskip
Therefore, from \eqref{projection ball} and the part {\it (ii)} in Lemma \ref{Nehari compactness}, the infimum \eqref{approximating inf} is attained at $v\in \overline{B_{s}(u)}\cap \MM_\varepsilon$. Finally, from \eqref{cruciale} we conclude that $v \in B_{s'}(u)$ and this completes the proof of Claim 3.

\medskip
Next, we prove \eqref{second ineq}. Notice first that for any $\varepsilon>0$ and any $v \in \MM_\ep$, $\tau_{0,v^+} \leq \tau_{\ep,v^+}= 1$ and $\tau_{0,v^-} \leq \tau_{\ep,v^-}=1$. Thus, the Mean Value Theorem and Claim 1, with $R=r$, imply that for any $\ep>0$ and any $v\in B_r(0)\cap \MM_{\ep}$,  
$$
\begin{aligned}
| J_0(\tau_{\varepsilon,v^+}v^+)-J_0(\tau_{0,v^+}v^+)| \leq & \max \limits_{t\in [\tau_{0,v^+},\tau_{\ep,v^+}]} \|DJ_0(t v^+)v^+\|_{(W^{1,p}_0(\Omega))^*}|\tau_{\ep,v^+} - \tau_{0,v^+}|\\
\leq & \max \limits_{t\in [\tau_{0,v^+},\tau_{\ep,v^+}]} K\ep^2
\end{aligned}
$$  
and  
$$
\begin{aligned}
| J_0(\tau_{\varepsilon,v^-}v^-)-J_0(\tau_{0,v^-}v^-)| \leq & \max \limits_{t\in [\tau_{0,v^-},\tau_{\ep,v^-}]} \|DJ_0(t v^+)v^+\|_{(W^{1,p}_0(\Omega))^*}|\tau_{\ep,v^-} - \tau_{0,v^-}|\\
\leq & \max \limits_{t\in [\tau_{0,v^-},\tau_{\ep,v^-}]} \|DJ_0(t v^+)v^+\|_{(W^{1,p}_0(\Omega))^*} K\ep^2.
\end{aligned}
$$  

Consequently, using part {\it (c)} in the last remark, we find $C_r>0$ such that for any $\ep>0$ and any $v \in \left(B_{s}(u)\setminus\overline{B_{s'}(u)}\right)\cap \MM_{\ep}$
$$
\begin{aligned}
|J_\varepsilon(v)-J_0(\tau_{0,v^+}v^++\tau_{0,v^-}v^-)|\leq &|J_\varepsilon(v)-J_0(v)|+|J_0(v)- J_0(\tau_{0,v^+}v^++\tau_{0,v^-}v^-)|
\\
\leq  & \varepsilon^2\int_{\Omega}|\nabla v|^2 dx \\
& \hspace{0.5cm}+| J_0(\tau_{\varepsilon,v^+}v^++\tau_{\varepsilon,v^-}v^-)-J_0(\tau_{0,v^+}v^++\tau_{0,v^-}v^-)|\\
\leq & C_r \ep^2.
\end{aligned}
$$	

Therefore, there exists $\hat{\ep}_2>0$ such that for $\varepsilon\in (0,\hat{\ep}_2)$ small enough such that, \eqref{second ineq} is satisfied.
		
\medskip

\bigskip	
{\bf Proof of {\it (i)}-{\it (iii)}.} Let $\ep_0 >0$ small be as in Claim 3. For any $\ep \in (0,\ep_0)$, Claim 3 and Lemma \ref{Deformation} implies that the infimum in \eqref{approximating inf} is actually attained by a critical point $u_{\ep}\in \MM_{\ep}\cap B_s(u)$ of $J_\varepsilon$.  

\medskip
On the other hand, from Theorem \ref{Theorem2}, $m_{\ep}(u_{\ep})=2$, $u_{\ep}$ is sign changing and $u_{\ep}$ has exactly two nodal domains. To finish the proof of this step, take $s=s_n$ in \eqref{approximating inf}, with $s_n \to 0^+$ as $n\to \infty$. Applying Claim 3 and the previous argument to the sequence $\{s_n\}_{n\in \N}$, we find the desired sequence $\{u_{\varepsilon_n}\}_{n\in \N}$. This concludes the proof of this step.

\bigskip
{\bf Proof of {\it (iv).}} Using parts {\it (i)}-{\it (iii)} from this theorem, we can select an approximating sequence $\{u_{\ep_n}\}_{n\in \N}\subset \MM_{\ep_n}$ with $u_{\ep_n} \to u_0$ and $\ep_n \to 0^+$, as $n\to \infty$. The rest of the proof follows the same lines as in the proof of Theorem \ref{Theorem3} with only slight and obvious changes. This completes the proof of {\it (iv)} and thus the proof of the theorem.
$ \blacksquare$
\end{proof}  

{\bf Remark:} The ideas of the proof of Theorem \ref{Theorem3} can be easily addapted to show that limit points of least energy nodal solutions in which the energy functional behaves like in the case $\ep \neq 0$ also have Morse index 2. More precisely, if $ w \in \MM_{0}$ with $J_0(w)=\alpha_0$ and there is a sequence of $w_n \in \MM_{\ep_n}$ such that:
\begin{itemize}
	\item[(i)] $\ep_n \to 0$ (with $\ep_n$ possibly 0).
	\item[(ii)] $J_{\ep_n}(w_n)=\alpha_n$.
	\item[(iii)] $w_n \to w$ strongly in $W_0^{1,p}(\Omega)$.
	\item[(iv)] $D^2J(w_n)|_{V_n}\geq 0$ with $V_n$ defined as in Theorem \ref{Theorem3} (including the case $\ep_n=0$).
\end{itemize}
Then we have that $m_0(w)=2$. In particular, combining this idea with Theorem \ref{Theorem4} we conclude that least energy nodal solutions in $\MM_0$ which are limit points of isolated least energy nodal solutions in $\MM_0$ also have Morse index 2.
\bigskip

%{\bf Acknowledgements.} \blue{Put here the projets, mentions and thanks you want to include. I have to include the same as in the Neumann paper (later) and possible the project that paid my visit to manizales.}

{\bf Acknowledgments.}

\end{document}